\documentclass[a4paper,12pt]{article}

\usepackage[utf8]{inputenc}

\usepackage{amsthm,amsmath,stmaryrd,bbm,hyperref,geometry,color,xcolor}
\usepackage{amssymb}
\usepackage[english]{babel}
\usepackage{graphicx}
\usepackage{amsfonts,amssymb}
\usepackage{verbatim}
\usepackage{enumitem}
\usepackage[all]{xy}
\usepackage{caption}
\usepackage{subcaption}

\usepackage{authblk}

\newcommand{\po}{\left(}
\newcommand{\pf}{\right)}

\newcommand{\cco}{\llbracket}
\newcommand{\ccf}{\rrbracket}
\newcommand{\E}{\mathbb E}
\newcommand{\R}{\mathbb R}

\newcommand{\D}{\mathcal D}
\newcommand{\X}{\mathbf X}
\newcommand{\x}{\mathbf x}

\newcommand{\N}{\mathbb N} 

\newcommand{\M}{\mathcal M}

\newcommand{\dd}{\mathrm{d}}

\newcommand{\argmin}{\operatornamewithlimits{argmin}}

\newtheorem{thm}{Theorem}
\newtheorem{assu}{Assumption}
\newtheorem{lem}[thm]{Lemma}

\newtheorem{prop}[thm]{Proposition}

\newtheorem{rem}{Remark}

\title{Condensation in Fleming--Viot particle systems with fast selection mechanism}

\author[1]{Lucas Journel\thanks{E-mail: lucas.journel@sorbonne-universite.fr}}
\author[2]{Tony Lelièvre\thanks{E-mail: tony.lelievre@enpc.fr}} 
\author[3]{Julien Reygner\thanks{E-mail: julien.reygner@enpc.fr}}

\affil[1]{LJLL, Sorbonne université, Paris, France}
\affil[2]{CERMICS, \'Ecole des Ponts and Inria, Marne-La-Vall\'ee, France.}
\affil[3]{CERMICS, École des Ponts, Marne-La-Vall\'ee, France}

\date{ }

\begin{document}

\maketitle

\begin{abstract}
  We study the Fleming--Viot particle system in a discrete state space, in the regime of a fast selection mechanism, namely with killing rates which grow to infinity. This asymptotics creates a time scale separation which results in the formation of a condensate of all particles, which then evolves according to a continuous-time Markov chain with jump rates depending in a nontrivial way on both the underlying mutation dynamics and the relative speed of growth to infinity of the killing rate between neighbouring sites. We prove the convergence of the particle system and completely describe the dynamics of the condensate in the case where the number of particles is kept fixed, and partially in the case when the number of particles goes to infinity together with (but slower than) the minimal killing rate.
\end{abstract}

\section{Introduction}\label{sec:motivation}

\subsection{Fleming--Viot particle systems}

\emph{Fleming--Viot particle systems}, sometimes also referred to as \emph{Moran processes}, are population dynamics models in which $n$ individuals are subjected to two mechanisms:
\begin{itemize}
  \item individuals evolve in some state space as independent copies of a given continuous-time stochastic process (called the \emph{mutation} process);
  \item they are killed at random times with a rate which depends on their current position, and are then instantaneously duplicated at the position of a uniformly chosen individual among the $n-1$ remaining ones (this is the \emph{selection}, or \emph{sampling}, mechanism). 
\end{itemize}
When the mutation process is a random walk on ${\mathbb Z}^d$ and the killing rate $\lambda>0$ is uniform on this space, the empirical measure of this particle system is known to converge, in the $n \to \infty$ limit and under a diffusive space-time rescaling for the mutation process, to the \emph{Fleming--Viot superprocess}~\cite{FleVio79}, which is a measure-valued process on $\R^d$ introduced to describe the distribution of alleles in a population with a large number of possible genetic states~\cite{Eth00}. In particular, in this regime, the selection mechanism is sped up by a factor $n$ with respect to the mutation process.

The asymptotic behaviour of the system is also of interest without space-time rescaling. Indeed, denoting by $(X_t)_{t \geqslant 0}$ the mutation process and $\tau$ the killing time, it is known that the empirical measure of the system generally converges, when $n \to \infty$, to the deterministic quantity $\mathbb{P}(X_t \in \cdot | \tau > t)$, which is the law of $X_t$ conditioned on survival up to time $t$. This observation, which also holds true for \emph{hard} killing mechanism, that is to say instantaneous killing when the process reaches some subset of the state space, has motivated numerous developments on the study of propagation of chaos and hydrodynamics limit~\cite{BurHolMar00,GriKan04,Lob09,AssFerGro11,Vil14,OcaVil17,FVsoft} as well as fluctuations~\cite{CerDelGuyRou20,LelPilRey18}, for both soft and hard killing mechanisms, in both discrete and continuous spaces.

From a numerical point of view, Fleming--Viot particle systems can be seen as Interacting Particle Systems~\cite{DelMic03,Del04,Rou06} allowing to approximate the conditional distribution  $\mathbb{P}(X_t \in \cdot | \tau > t)$, whose evolution is given by a nonlinear Fokker--Planck equation and is therefore nontrivial to study. The $t \to \infty$ limit of this conditional distribution is called a \emph{quasistationary distribution} (QSD) for the process $(X_t)_{t \geqslant 0}$~\cite{ColMarSan13}, and plays a prominent role in population dynamics~\cite{MelVil12}, molecular dynamics~\cite{LeBLelLusPer12,DiGLelLePNec17} or Monte Carlo methods~\cite{PolFeaJohRob20}. For recent works on the existence and uniqueness of QSDs, as well as convergence of the conditional distribution $\mathbb{P}(X_t \in \cdot | \tau > t)$ towards the QSD, we refer to~\cite{champagnat2018general,CV2020,bansaye-Cloez-Gabriel-Marguet}.

While the numerical interest of Fleming--Viot particle systems lies in the consistency of their $n \to \infty$ limits, in practice they can only be implemented with finitely many particles. In this context, one may expect their efficiency to depend on the intensity of the killing rate. More precisely, to approximate the QSD of the mutation process by the stationary distribution of the Fleming--Viot particle process, one needs the latter to reach stationarity on a shorter time scale than the typical killing time. This time scale separation is referred to as the \emph{metastability} of the system~\cite{OliVar05}; in this situation, quantitative estimates for Fleming--Viot particle systems associated with diffusions with hard killing were recently obtained in~\cite{FVhard}. The primary motivation of the present article is to study the behaviour of Fleming--Viot particle systems in \emph{non-metastable} systems, namely for which killing occurs on a faster time scale than convergence to stationarity for the mutation process. We therefore consider the Fleming--Viot particle system in the asymptotic regime where the mutation dynamics remains fixed, but the killing rates grow to infinity uniformly, with a minimal rate $\underline\lambda$. 


\subsection{The fast selection asymptotics}

In the mathematical biology literature, and more precisely in the study of evolutionary models, this time scale separation between slow mutation and fast selection is known to induce the convergence of microscopic Individual Based Models toward models of \emph{monomorphic} populations, namely populations in which all individuals share the same phenotype, or trait~\cite{Cha06,ChaLam07}. The dynamics of the dominant trait, on the slow time scale, is then driven by mutations, which after an infinitesimally short competition, invade the population and replace the previous dominant trait. From the mathematical point of view, this is expressed by the convergence of the empirical distribution of the microscopic system to a Dirac mass. The limiting dynamics of this Dirac mass, which therefore describes the consecutive dominant traits of the population, is then the main object of interest.

We prove  that, in the fast selection asymptotics, the same phenomenon occurs for the Fleming--Viot particle system: its empirical measure converges to a Dirac mass, which evolves randomly on the slow time scale. We establish this \emph{condensation} phenomenon, and show that the dynamics of the Dirac mass (which we shall call \emph{condensate}) is determined in a nontrivial way by both the jump rates of the microscopic mutation dynamics, and by the relative speed of growth to infinity of the killing rates between neighbouring sites. We shall first address the case of a fast selection and constant number of particles, and then the case of both fast selection and large number of particles. In the context of convergence of microscopic Individual Based Models, these two asymptotic regimes were addressed in~\cite{ChaLam07} and~\cite{Cha06}, respectively. The convergence of a Fleming--Viot system toward a Dirac mass evolving according to a deterministic ordinary differential equation was also recently derived in~\cite{Champagnat-Hass}, however in the different asymptotic of a \emph{small} (but not slow) mutation. More generally, the condensation phenomenon is also known to occur in various interacting particle systems in statistical mechanics, such as the Zero Range Process~\cite{GroSchSpo03}, or the reversible Inclusion Process~\cite{BiaDomGia17}, in the regime of a large number of particles.

\subsection{Outline of the article}

The article is organised as follows. In Section~\ref{sec:main_results}, the precise mathematical setting is introduced, and the results are presented. The proofs of Theorems~\ref{thm:soft-kill-1} and~\ref{thm:conv-trajectories} which give the asymptotic behaviour in the high killing rate regime for a fixed number of particles are respectively presented in Sections~\ref{sec:general-rate} and~\ref{sec:conv-trajectories}. Finally, the proof of Theorem~\ref{thm:soft-kill-2} which describes the asymptotic dynamics in the limit where both the killing rate and the number of particles go to infinity is provided in Section~\ref{sec:rate-bounded-var}.

\section{Setting and main results}\label{sec:main_results}

After presenting the mathematical setting in Section~\ref{sec:MS}, we will present our results first for a fixed number of particles in Section~\ref{ss:res-fixed} and then for infinitely many particles in Section~\ref{ss:res-infty}.

\subsection{Mathematical setting}\label{sec:MS}

We consider the Fleming--Viot particle process associated with a mutation process which is a continuous-time Markov chain in some finite or countably infinite state space $D$, with jump rates $(q(x,y))_{x,y \in D}$, and a selection mechanism induced by killing rates $(\lambda(x))_{x \in D}$. We suppose that
\begin{equation}\label{eq:assQlambda}
    Q := \sup_{x\in D}\sum_{y\in D} q(x,y) < \infty, \qquad \sup_{x\in D}\lambda(x) < \infty.
\end{equation}
We denote by $\X = (\X_t)_{t \geqslant 0}$ the Fleming--Viot $n$-particle process in $D^n$. We are interested in the evolution of the empirical measure $\pi(\X_t)$ of the system, where for any vector $\x=(x_1, \ldots, x_n) \in D^n$,
\begin{equation}\label{eq:pi}
\pi(\x)=\frac{1}{n}\sum_{i=1}^n\delta_{x_i}.
\end{equation}
Since the particles of the Fleming-Viot process are exchangeable, the Markov dynamics can be expressed in terms of the empirical measure $\pi(\X_t)$.
We write $\M^1(D)$ for the set of probability measures on $D$ and $\M^1_n(D) \subset \M^1(D)$ for the set of empirical measures of $n$ particles in $D$:
\[
\mathcal M^1_n(D) = \{\pi(\x), \x \in D^n\} = \left\{ \po\frac{k_x}{n}\pf_{x\in D}\middle|\forall x \in D, k_x\in \N \text{ and } \sum_{x\in D} k_x = n \right\}.
\] 
In this space, the empirical measure $(\pi(\X_t))_{t \ge 0}$ of the Fleming--Viot process is the continuous-time Markov chain with infinitesimal generator 
\[L =L_m + L_s,\]
where $L_m$ and $L_s$ respectively correspond to mutation and selection. More precisely, introducing the notation 
\[
\ell^{\infty}( \mathcal M^1_n(D)) = \left\{ u : \mathcal M^1_n(D)\to \R, \|u\|_{\infty} < \infty \right\},\quad 
\|u\|_{\infty} = \sup_{\xi\in  \mathcal M^1_n(D)}|u(\xi)|,
\]
the operator $L_m$ is the infinitesimal generator of the empirical measure of $n$ independent random walks in $D$ with jump rates $q(x,y)$: for all $f \in \ell^{\infty}( \mathcal M^1_n(D))$, for all $\xi \in \mathcal M^1_n(D)$,
\begin{equation}\label{eq:generateur-diffusion}
L_mf(\xi) = \sum_{x,y\in D} n\xi(x) q(x,y) \left( f\left(\xi +\frac{\delta_y - \delta_x}{n}\right) - f(\xi) \right).
\end{equation}
On the other hand, the operator $L_s$ describes the evolution of the empirical measure of $n$ particles experiencing only the selection mechanism, where a particle located in $x$ is killed at rate $\lambda(x)$ and duplicated at the location of one of the $n-1$ remaining particles, chosen uniformly: for all $f \in \ell^{\infty}( \mathcal M^1_n(D))$, for all $\xi \in \mathcal M^1_n(D)$,
\[
L_sf(\xi) = \sum_{x,y\in D} \frac{n^2}{n-1} \xi(x)\lambda(x)\xi(y) \left( f\left(\xi +\frac{\delta_y - \delta_x}{n}\right) - f(\xi) \right).
\]
 
Under~\eqref{eq:assQlambda}, the Markov generators $L$, $L_m$ and $L_s$ are bounded operators of $\ell^{\infty}( \mathcal M^1_n(D))$. Moreover, the latter space is put in duality with the space
\[
\ell^1( \mathcal M^1_n(D)) = \left\{ u : \mathcal M^1_n(D)\to \R, \|u\|_{1} < \infty \right\},\quad 
\|u\|_{1} = \sum_{\xi\in  \mathcal M^1_n(D)}|u(\xi)|,
\]
through the duality bracket
\[
\forall u \in \ell^1(\mathcal M^1_n(D)), v\in \ell^{\infty}(\mathcal M^1_n(D)), \qquad \langle u,v \rangle = \sum_{\xi \in \mathcal M^1_n(D)}u(\xi)v(\xi).
\]
We denote by $L^*$, $L_m^*$ and $L_s^*$ the respective adjoint operators of $L$, $L_m$ and $L_s$ with respect to this duality bracket. They are bounded operators of $\ell^1(\mathcal M^1_n(D))$.

\medskip 

As explained in Section~\ref{sec:motivation}, we are interested in the 
regime of a fast selection mechanism, namely 
\begin{equation}\label{eq:context}
     \underline \lambda := \inf_{x \in D} \lambda(x) \rightarrow \infty.
\end{equation}
A key remark in our approach is that, under the selection dynamics, particles can only move to a site which already contains other particles, so that Dirac masses $\delta_x \in \mathcal{M}^1_n(D)$, $x \in D$ are absorbing states for this dynamics. In the context of evolutionary models, the absorption time of the selection mechanism is called the \emph{fixation time}. Therefore, if this time occurs faster than the typical time between two consecutive mutations, on the time scale of the mutation dynamics, all particles condensate in a single site, and the motion of the condensate, in the set
\begin{equation}\label{eq:Dirac_set}
\Delta = \left\{ \delta_x, x\in D\right\}
\end{equation}
of Dirac masses on $D$, is triggered by mutations. The goal of this article is precisely to describe the dynamics of this condensate. We first address the case where $n$ is kept fixed while $\underline\lambda$ satisfies~\eqref{eq:context} in Section~\ref{ss:res-fixed}, and then the case where both $n$ and $\underline\lambda$ go to infinity in Section~\ref{ss:res-infty}. In the former case, we importantly use the fact that with probability going to~$1$, the fixation time always (on a finite time horizon) occurs before a new mutation. This allows us to employ functional analytic tools of averaging for slow/fast systems~\cite{Pavliotis-Stuart}, but in the somehow not so common case where the fast dynamics has several absorbing states rather than a unique stationary distribution. In the case where both $n$ and $\underline\lambda$ go to infinity, we directly show that the order of magnitude of absorption under the selection mechanism is $\underline \lambda/n$, and therefore that, on the time scale of the mutation dynamics, the condition
\begin{equation}\label{eq:time-scale-separation}
    \frac{\underline \lambda}{n} \gg 1
\end{equation}
ensures condensation.

\subsection{Fixed number of particles}\label{ss:res-fixed}

In this section, the number of particles $n$ is fixed, and we show that, in the limit~\eqref{eq:context}, the trajectories of the empirical process $(\pi(\X_t))_{t \ge 0}$ converge towards trajectories in the set $\Delta$ defined in~\eqref{eq:Dirac_set}.
This limiting dynamics can be precisely described as a continuous-time Markov chain in $D$. 

In order to make precise the limiting regime~\eqref{eq:context}, let us consider a family $(\lambda_r)_{r\in\N}$ of functions $D \to \R_+$, and introduce the notation
\begin{equation}\label{eq:inf-lambda}
    \underline \lambda_r = \inf_{x \in D} \lambda_r(x).
\end{equation}
We set
\begin{equation}\label{eq:coef_lim}
\alpha_{x,y,r} = \frac{\lambda_r(y)}{\lambda_r(x)},
\end{equation}
and we denote by $\X^r = (\X^r_t)_{t \geqslant 0}$ the underlying Fleming--Viot $n$-particle processes.
We work in this section in the following setting:
\begin{assu}\label{assu:assu1}
The number $n$ of particles is fixed, $\lim_{r\rightarrow\infty}\underline\lambda_r = +\infty$, and
for all $x,y \in D$, there exists $\alpha_{x,y,\infty}\in [0,\infty]$ such that:
\begin{equation}\label{eq:lim_coef}
\lim_{r\rightarrow\infty}\alpha_{x,y,r}=\alpha_{x,y,\infty}.
\end{equation}
\end{assu}

As explained above, 
it is expected that $(\pi(\X^r_t))_{t \ge 0}$ converges to a process $(\delta_{Y_t})_{t \ge 0}$ where $(Y_t)_{t \ge 0}$ is a continuous-time Markov chain on $D$. 
Intuitively, the evolution of $(Y_t)_{t \ge 0}$ is driven by the following process. If all the particles are in $x\in D$ at a given time $t$ (i.e. $Y_t=x$), then the rate at which one particle goes from $x$ to $y$ is given by $nq(x,y)$. After such a move of one particle from $x$ to $y$, since 
the selection mechanism occurs on a faster time scale than the mutation process, all the particles will condensate on one of the two states~$x$ or~$y$ before any other jumps occurs. We shall show in Lemma~\ref{lem:game} below that the probability that all particles condensate on the new site $y$ is given by a Gambler's Ruin problem, and is equal to $(\alpha_{x,y,r}-1)/(\alpha_{x,y,r}^n-1)$. This justifies the introduction of the new jump rates:
\begin{equation}\label{eq:jump-rates-r}
    \tilde q_{r}(x,y) = \left\{ \begin{aligned} n &q(x,y) \frac{\alpha_{x,y,r}-1}{\alpha_{x,y,r}^n-1} &\text{ if }\alpha_{x,y,r} \neq 1, \\ &q(x,y)&\text{ otherwise,} 
    \end{aligned}\right.
\end{equation}
as well as their $r \to \infty$ limit:
\begin{equation}\label{eq:jump-rates-lim}
    \tilde q_{\infty}(x,y) = \left\{ \begin{aligned} n &q(x,y) \frac{\alpha_{x,y,\infty}-1}{\alpha_{x,y,\infty}^n-1} &\text{ if }\alpha_{x,y,\infty} \neq 1, \\ &q(x,y)&\text{ otherwise.} 
    \end{aligned}\right.
\end{equation}
In the above, $\tilde q_{\infty}(x,y)=0$ in the case $\alpha_{x,y,\infty} = \infty$. 

\medskip
Our first theorem concerns the limit of the time-marginal distribution of the empirical measure process under Assumption~\ref{assu:assu1}. For any random variable $\pi$ in $\mathcal{M}^1_n(D)$, we identify the probability measure $\mathcal{L}aw(\pi)$ as an element of $\ell^1(\mathcal{M}^1_n(D))$ and therefore use the $\ell^1$ norm to measure distances between such probability measures. 
 
\begin{thm}\label{thm:soft-kill-1}
Under Assumption~\ref{assu:assu1}, for all $\eta\in\mathcal M^1_n(D)$, there exist a (deterministic) probability measure $\eta_\infty\in \mathcal M^1(D)$ described in Lemma~\ref{lem:initial-condi} and a random variable $Y_0$ of law $\eta_\infty$, such that if $(Y_t)_{t \geqslant 0}$ is a Markov process on $D$ with jump rate $(\tilde q_{\infty}(x,y))_{(x,y)\in D}$ and initial condition $Y_0$, then: for all $t>0$,
\[
\lim_{r\rightarrow\infty} \| \mathcal Law\po \pi(\X^r_t) \pf - \mathcal Law\po \delta_{Y_t} \pf\|_1=0.
\]
Moreover, if $\pi(\X^r_0)=\delta_z$ for some $z\in D$, then for all $T>0$,  there exists $C>0$ such that if $(Y_t^r)_{t \geqslant 0}$ is a Markov process on $D$ with jump rate $(\tilde q_{r}(x,y))_{(x,y)\in D}$ and initial condition $Y^r_0=z$, then: 
\[
\sup_{t \in [0,T]} \| \mathcal Law\po \pi(\X^r_t) \pf - \mathcal Law\po \delta_{Y_t^r} \pf\|_1 \leqslant \frac{C}{\underline \lambda_r}.
\]
\end{thm}

In the first statement of the theorem, the marginal distribution is taken at time $t>0$. Indeed, if the Fleming--Viot particle system is initialized with a distribution which is not a Dirac distribution, namely with particles spread among several sites, then on the short time scale $1/\underline\lambda_r$, the selection mechanism makes all particles condensate in a single site, whose distribution is precisely the probability measure $\eta_\infty$. This shows that, in general, the convergence of $\mathcal Law\po \pi(\X^r_t) \pf$ to $\mathcal Law\po \delta_{Y_t}\pf$ cannot hold uniformly on $[0,T]$. However, if the initial distribution of the Fleming--Viot particle system is already condensed on a single site, then the second statement of the theorem shows that the distance between $\mathcal Law\po \pi(\X^r_t) \pf$ and $\mathcal Law\po \delta_{Y^r_t}\pf$ may be controlled uniformly on $[0,T]$, with a quantitative bound which is obtained by using averaging methods, as presented in~\cite[Chapter 16]{Pavliotis-Stuart}. The remaining distance, between $\mathcal Law\po \delta_{Y^r_t}\pf$ and $\mathcal Law\po \delta_{Y_t}\pf$, also vanishes (see Lemma~\ref{lem:conv-Yr-to-Y}), with a rate which depends on the convergence of $(\alpha_{x,y,r})_{x,y \in D}$ toward $(\alpha_{x,y,\infty})_{x,y \in D}$. The proof of Theorem~\ref{thm:soft-kill-1} is given in Section~\ref{sec:general-rate}. 

\medskip
We now state a trajectorial convergence result for the Fleming--Viot particle system towards the $\Delta$-valued dynamics $(\delta_{Y_t})_{t \geqslant 0}$. For reasons which are detailed below, the Skorohod topologies for cadlag processes are not adapted to this system, and therefore we need to introduce another topology on the space of sample paths. To proceed, we first introduce the total variation norm $\|\cdot\|_{TV}$, and defined by: for $\nu,\mu\in \mathcal M^1(D)$,
 \[
 \|\mu-\nu\|_{TV} = \sum_{x\in D}|\mu(x) - \nu(x)|.
 \]
Note that, since $D$ is a finite or countably infinite set, the weak topology on $\mathcal M^1(D)$ is metrized by the total variation norm. Next, for any $T>0$, let $\mathcal D([0,T],\mathcal M^1(D))$ be the set of cadlag functions from $[0,T]$ to $\mathcal M^1(D)$. For $\pi,\tilde \pi \in \mathcal D([0,T],\mathcal M^1(D))$, we define the $L^1_{TV}$ distance by:
\begin{equation}\label{eq:distance-L1}
\|\pi - \tilde \pi\|_{L^1_{TV}}  =  \int_0^T \|\pi_t(x) - \tilde \pi_t(x)\|_{TV} \dd t = \int_0^T \sum_{x\in D} |\pi_t(x) - \tilde \pi_t(x)| \dd t.
\end{equation}
This defines a distance on $\mathcal D([0,T],\mathcal M^1(D))$, and the associated topology coincides in our case with the so-called Meyer-Zheng topology~\cite{meyer1984tightness}. We then have:

\begin{thm}\label{thm:conv-trajectories}
Under Assumption~\ref{assu:assu1}, for all $T>0$, the process $\po\pi(\X^r_t)\pf_{t\in [0,T]}$ converges in distribution, in $L^1_{TV}$, to the process $\po \delta_{Y_t}\pf_{t\in [0,T]}$ introduced in Theorem~\ref{thm:soft-kill-1}.
\end{thm}
The proof of this theorem is given in Section~\ref{sec:conv-trajectories}. 

\medskip
Let us conclude this section with two remarks. First, due to the nature of the process, the convergence in Theorem~\ref{thm:conv-trajectories} does not hold in the Skorokhod topology based on $L^\infty$ norms, but in the $L^1_{TV}$ topology which is a weaker topology. Indeed, when a particle jumps from $x$ to $y$, one observes motions from $x$ to $y$ and then back to~$x$ due to the selection mechanism,  over a time scale of order $1/\underline{\lambda}_r$. This implies jumps of non vanishing heights (multiple of $1/n$) in the process $\pi(\X_t^r)$, and therefore, one cannot expect convergence to hold in the Skorokhod topology. Second, since the mapping $\pi \in \mathcal{D}([0,T],\mathcal M^1(D)) \mapsto \pi_t \in \mathcal M^1(D)$ is not continuous with respect to the $L^1_{TV}$ distance,
the convergence in Theorem~\ref{thm:conv-trajectories} 
does not imply the convergence of the time-marginal distribution from Theorem~\ref{thm:soft-kill-1}.

\subsection{Infinitely many particles}\label{ss:res-infty}

In this section, we are interested in the asymptotic regime where both the number of particles and the killing rates go to infinity. Therefore, in addition to the $r$-indexed family $(\lambda_r)_{r\in\N}$ of death rates, we introduce $(n_r)_{n\in\N} \subset \N^{\N}$ a sequence of numbers of particles. In order to keep the time scale separation~\eqref{eq:time-scale-separation} between the mutation process and the selection mechanism, 
we impose 
that $\underline \lambda_r/n_r \to \infty$, where $\underline \lambda_r$ remains defined by~\eqref{eq:inf-lambda}. Notice that this condition is necessary for the condensation phenomenon to occur: indeed, as already mentioned in Section~\ref{sec:motivation}, in the case $\underline{\lambda}_r = cn_r$, $\pi(\X^r_t)$ converges under a proper rescaling to the Fleming--Viot super-process and therefore exhibits a diffusive behaviour, see~\cite{Eth00}.

Last, we let $(\X_0^r)_{r\in\N}$ be a sequence of initial conditions, such that $\X_0^r\in D^{n_r}$, and we will again denote by $\X^r = (\X^r_t)_{t \geqslant 0}$ the associated Fleming--Viot $n_r$-particle process.

\subsubsection{Nearly homogeneous killing rates}
Compared to the previous setting, we restrict ourselves to a situation where $\lambda_r(x)$ is a bounded perturbation of $\underline \lambda_r$. In this regime, the quantities $\alpha_{x,y,\infty}$ defined in the previous section are all equal to $1$, and therefore the jump rates $\tilde{q}_\infty(x,y)$ in~\eqref{eq:jump-rates-lim} are equal to $q(x,y)$ whatever the value of $n$. As a consequence, letting $n \to \infty$, one may expect the limiting jump rates to be simply $q(x,y)$, and this will indeed be the case as stated in Theorem~\ref{thm:soft-kill-2} below. 

Here is the precise setting of this section:
\begin{assu}\label{assu:assu2}
\[
\sup_{r \in\N} \sup_{x,y\in D} |\lambda_r(x) - \lambda_r(y)| < \infty, \quad
\lim_{r \rightarrow \infty} \frac{\underline \lambda_r}{n_r} = \infty,
\quad \text{and} \quad
\exists \eta\in\mathcal M^1(D), \, \pi(\X_0^r) \underset{r\rightarrow\infty}{\rightarrow}\eta \text{ weakly}.
\]
\end{assu}

\begin{thm}\label{thm:soft-kill-2}
Let us assume that Assumption~\ref{assu:assu2} holds.
 Let $(Y_t)_{t \geqslant 0}$ be a Markov process on $D$ with jump rate $( q(x,y))_{(x,y)\in D}$ and initial condition $\eta$. Then, the following convergence in law holds: for all $t>0$,
\[
\pi(\X^r_t) \underset{r\rightarrow\infty}{\rightarrow} \delta_{Y_t},
\]
where the convergence in law is considered with respect to the topology defined by the total variation distance on $\M^1(D)$.
Moreover, if $\pi(\X_0^r)=\delta_z$ for some $z\in D$ and for all $r\in\N$, then for all $T>0$, there exists $C>0$ such that for all $F:\mathcal{M}^1(D)\to \R$ bounded and Lipschitz continuous, and for all $t \in [0,T]$:
\[
\left|\E\po F\po\pi(\X^r_t) \pf\pf - \E\po F\po\delta_{Y_t} \pf\pf \right| \leqslant C\po\|F\|_{\mathrm{Lip}} + \|F\|_{\infty}\pf\sqrt{\frac{n_r}{\underline \lambda_r}}.
\]
\end{thm}

\begin{rem}
    Since the convergence in law of $\pi(\X^r_t)$ toward $\delta_{Y_t}$ in Theorem~\ref{thm:soft-kill-2} is obtained for any sequence of initial conditions $\pi(\X^r_0)$ which converge weakly to $\eta$, it is easily seen that this result actually implies the convergence in the sense of the finite-dimensional distributions (in time) of the process $(\pi(\X^r_t))_{t \geqslant 0}$ to $(\delta_{Y_t})_{t \geqslant 0}$.
\end{rem}

The proof is completely different in nature from the ones of Theorems~\ref{thm:soft-kill-1} and~\ref{thm:conv-trajectories}, and relies on the Kolmogorov equation associated with the dynamics on $(\pi(\X_t^r))_{t \ge 0}$, see Section~\ref{sec:rate-bounded-var}. Notice that
the speed of convergence in Theorem~\ref{thm:soft-kill-2} is not optimal. Indeed, choosing $n_r=n$ constant (which is allowed in Assumption~\ref{assu:assu2}), Theorem~\ref{thm:soft-kill-2} yields a rate of convergence of order $1/\sqrt{\underline{\lambda}_r}$ whereas we know from Theorem~\ref{thm:soft-kill-1} that it is of order $1/\underline{\lambda}_r$.  

\subsubsection{A conjecture for general killing rates}

The objective of this section is to provide a conjecture on the limiting behaviour of $(\pi(\X^r_t))_{t \geqslant 0}$ in the general case where both $n_r$ and $\underline\lambda_r$ go to infinity in the regime~\eqref{eq:time-scale-separation}, but without assuming that $\lambda_r(x)$ is a bounded perturbation of $\underline\lambda_r$ (namely without the first hypothesis in Assumption~\ref{assu:assu2}). We conjecture that in this case the limiting behaviour of $(\pi(\X^r_t))_{t \geqslant 0}$ remains described by some process of the form $(\delta_{\overline Y_t})_{t \geqslant 0}$, where the dynamics of $(\overline Y_t)_{t \geqslant 0}$ is obtained by taking the $n \to \infty$ limit of the dynamics of the process $(Y_t)_{t \geqslant 0}$ identified in Theorem~\ref{thm:soft-kill-1}. The limits of the associated jump rates $\tilde{q}_\infty(x,y)$ defined in~\eqref{eq:jump-rates-lim} write
\begin{equation}
    \lim_{n\rightarrow\infty}\tilde q_{\infty}(x,y) = \left\{ \begin{aligned}  &q(x,y) &\text{ if }\alpha_{x,y,\infty} =1, \\ &\infty&\text{ if } \alpha_{x,y,\infty}<1, q(x,y)\neq 0, \\ & 0 &\text{ otherwise,}
    \end{aligned}\right.
\end{equation}
so the limiting dynamics is not obvious because of the possibility for some jump rates to take the value $\infty$. However, it can be understood as follows: in the limit $n\rightarrow\infty$, $Y_t$ will not stop on a site $x\in D$ that is connected to at least one other site $y\in D$ with a smaller killing rate, namely $\alpha_{x,y,\infty}<1$. In the case where there are several such $y$, the next site will be chosen among the ones with minimal death rate. 
Let us make this precise.

First the state space of $(\overline Y_t)_{t \ge 0}$ is:
\[
\mathfrak D = \left\{x\in D : \text{for all $y \in D$ such that $q(x,y)>0$, we have $\alpha_{x,y,\infty}\geqslant 1$}\right\}.
\]
From $x \in \mathfrak D$, the chain can make two kinds of jumps. First it can move to a point $y \in \mathfrak D$ such that $q(x,y)>0$ and $\alpha_{x,y,\infty}=1$, and this happens with a rate $q(x,y)$. Second, it can move to a point $y \in \mathfrak D$ accessible by visiting a point $z$ such that $q(x,z)>0$, $\alpha_{x,z,\infty}=1$, and $z$ is connected to at least one point with a smaller killing rate. In particular, $z\notin \mathfrak D$. In order to describe this second type of jump, let us define, for any $z \not\in \mathfrak D$ and $y \in \mathfrak D$, the set $\mathbb D_{z,y}$ of sequences $(z_0,\dots,z_k) \in D^{k+1}$ such that, for all $i \in \llbracket 0,k-1\rrbracket$, $q(z_i,z_{i+1})>0$, $\alpha_{z_{i},z_{i+1},\infty}<1$, and $\alpha_{z_{i+1},\tilde z,\infty} \geqslant 1$ for any $\tilde z$ such that $q(z_{i},\tilde z)>0 $. We next introduce the accessible points $y \in \mathfrak D$ from such a $z$:
\[\mathbb C_z = \left\{y \in \mathfrak D | \text{$\mathbb D_{z,y}$ is non empty}\right\}.\]
The chain $\overline Y$ can make a jump of the second type if the set
\[
\mathcal Z_{x,y} = \left\{ z\in D, \alpha_{x,z,\infty} = 1, q(x,z)>0, y\in  \mathbb C_z\right\}
\]
is non-empty. To compute the probability to go from $x\in \mathfrak D$ to $y\in \mathfrak D$ in this case, let us introduce:
\[
\mathbb C_z^1 = \left\{y \in  D| q(z,y)>0, \alpha_{z,y,\infty}<1\,\text{ and } \alpha_{y,\tilde z,\infty} \geqslant 1, \forall \tilde z , q(z,\tilde z)>0 \right\}.
\]
If $\mathbb C^1_z$ is non-empty, then one simply has $\mathbb C_z^1  =  \argmin_{y:q(z,y)>0} \alpha_{z,y,\infty}$. Besides, for $z\in \mathcal Z_{x,y}$ where $y\in \mathfrak D$, and for a sequence $z_0,\cdots,z_k\in \mathbb D_{z,y}$, one has $z_{i+1} \in \mathbb C^1_{z_i}$ for all $i \in \llbracket0,k-1\rrbracket$.
We are now in position to write that the probability to go instantaneously from $z\in \mathcal Z_{x,y}$ to $y\in \mathfrak D$ along some sequence $z_0,\cdots,z_k\in \mathbb D_{z,y}$ is:
\[
\prod_{i=0}^{k-1} \frac{q(z_{i},z_{i+1})}{\sum_{z\in \mathbb C_{z_{i}}^1}q(z_{i},z)}.
\]
Putting everything together, the expected jump rates are for all $x,y\in \mathfrak D$:
\[
q_{\infty}^{\infty}(x,y) = q(x,y) \mathbbm 1_{\alpha_{x,y,\infty}=1} + \sum_{z\in \mathcal Z_{x,y}} q(x,z) \sum_{(z,z_1,\cdots,z_k)\in \mathbb D_{z,y}} \prod_{i=0}^{k-1} \frac{q(z_{i},z_{i+1})}{\sum_{z\in \mathbb C_{z_{i}}^1}q(z_{i},z)}.
\]
We were however not able to show this result. First, the method used in the proof of Theorem~\ref{thm:soft-kill-1} fails in the case $n\rightarrow\infty$ as many estimates yield useless results in the limit of an infinite number of particles. Similarly, the proof of Theorem~\ref{thm:soft-kill-2} cannot deal with killing rates that are not of the same order of magnitude, as the proof of Lemma~\ref{lem:high-corr} requires $r\mapsto \sup_{x,y}|\lambda_r(x)-\lambda_r(y)|$ to be bounded.

\section{Proof of Theorem~\ref{thm:soft-kill-1}}\label{sec:general-rate}

We let Assumption~\ref{assu:assu1} hold in all this section. For all $x\in D$, write:
\begin{equation}\label{eq:normalised-coef}
    \alpha_r(x) = \frac{\lambda_r(x)}{\underline \lambda_r}.
\end{equation}
For all $x\in D$, $\alpha_r(x)\geqslant 1$, and we write:
\begin{equation}\label{eq:decompos_L}
L = L_m + \underline \lambda_r L_{s,0}^{r},
\end{equation}
where for $f \in \ell^\infty(\mathcal M^1_n(D))$ we denoted:
\begin{equation}\label{eq:Ls0}
L_{s,0}^{r}f(\xi) = \sum_{x,y\in D} \frac{n^2}{n-1} \xi(x)\alpha_r(x)\xi(y) \left( f\left(\xi +\frac{\delta_y - \delta_x}{n}\right) - f(\xi) \right) =\frac{1}{\underline \lambda_r}L_sf(\xi).
\end{equation}
For simplicity, we drop the dependence in $r$ of $L_{s,0}^{r}$ and we simply write $L_{s,0}$.

The proof of Proposition~\ref{thm:soft-kill-1} uses averaging methods~\cite[Chapter 16]{Pavliotis-Stuart}. Let us first present the overall idea. Recall the notation~\eqref{eq:pi} for the empirical measure. Denote for all $t \ge 0$ and  $z\in D$
\begin{equation}\label{eq:empiri-meas}
 \pi_t(z) =\pi(\X_t^r)(z),
\end{equation}
and let $u$ be the law of the empirical measure process $(\pi_t)_{t \ge 0}$:
\begin{equation}\label{eq:u}    
u(t,\xi) = \mathbb P\po \pi_t = \xi \pf,\qquad \forall \xi \in \mathcal M^1_n(D). 
\end{equation}
For all $t\geqslant 0$, $u(t,\cdot)\in \ell^1(\mathcal M^1_n(D))$, and we have the Kolmogorov equation:
\[
\partial_t u = L^*u, 
\]
where 
we recall that $L^*$ acts on $\ell^1(\mathcal M^1_n(D))$. 

Let us present a formal argument in order to identify the limiting dynamics. If we assume an expansion of the form 
\begin{equation*}\label{eq:Expansion}
u = u_0 + \frac{1}{\underline\lambda_r}u_1 + o\po \frac{1}{\underline\lambda_r} \pf,
\end{equation*}
with $\|u_0\|_{1}$ and $\|u_1\|_{1}$ bounded uniformly in $r\in \N$ and $t\geqslant 0$, then, using~\eqref{eq:decompos_L}, a formal identification at order $({\underline\lambda_r})^1$ yields that $u_0$ satisfies $$L_{s,0}^{*}u_0=0,$$ which shows that one expects the limiting dynamics to only take Dirac mass values. Indeed, for all $r\in\N$, $L_{s,0}$ is the generator of a Markov process that corresponds to a Fleming--Viot process  with no mutation mechanism and only selection dynamics. We will refer in the sequel to this dynamics as the selection dynamics. Under this dynamics, each particle of the Fleming--Viot process is killed at rate $\alpha_r$, and then branches uniformly at random over the surviving particles. Let $(\pi^0_t)_{t\geqslant 0}$ be a Markov process with this generator $L_{s,0}$. Every Dirac measure is an absorbing point of such dynamics, and thus 
${\rm Span} \po\{ \mathbbm 1_x, x \in D\} \pf \subset {\rm Ker} L^*_{s,0}$, where
\begin{equation}\label{eq:dirac-dirac}
\mathbbm 1_x = \delta_{\delta_x}
\end{equation}
is a probability measure on $\mathcal M^1_n(D)$, and can also be seen as an element of $\ell^1(\mathcal M^1_n(D))$.
Actually,  writing (remember the notation~\eqref{eq:Dirac_set})
\begin{equation}\label{eq:tau_urne}
\tau = \inf\left\{ t\geqslant 0, \pi^0_t \in \Delta \right\},
\end{equation}
we will show that $\tau <\infty$ almost surely (as an application of Lemma~\ref{lem:different-time-scale} below), and this will allow us to prove that, as a bounded operator from $\ell^1\po \mathcal M^1_n(D) \pf$ to itself, $L_{s,0}^*$ satisfies
\[
\overline{{\rm Span} \po\{ \mathbbm 1_x, x \in D\} \pf} = {\rm Ker} L^*_{s,0},
\]
see Lemma~\ref{lem:topology}.

We will then postulate the limiting dynamics on $u_0$ using our intuition on how Dirac masses evolve. Let us recall the intuition to derive these limiting dynamics: if the Dirac mass is at $x\in D$, at rate $nq(x,y)$ one particle will go from $x$ to $y$. Since the selection dynamics is on a much faster time scale, we only have to consider this part of the dynamics, and the probability that the Dirac mass goes from $x$ to $y$ is given by the probability $p_{x,y}$ of a Gambler's ruin, see Lemma~\ref{lem:game}. Therefore, the rates at which the Dirac mass will move from $x$ to $y$ is simply $nq(x,y)p_{x,y}$.

To formalize this reasoning, after having introduced $u_0$ following the expected limiting dynamics described above, we will define $u_1$ as a solution to the Poisson equation
\begin{equation}\label{eq:Poisson}
L_{s,0}^* u_1 = \partial_t u_0 - L_m^*u_0
\end{equation}
which is the equation formally obtained at order $({\underline\lambda_r})^0$. The proof will then consist in proving that 
\[
u - u_0 - \frac{1}{\underline\lambda_r}u_1
\]
goes to $0$ as $r\rightarrow \infty$.

The remaining of this section is divided into four parts. In Section~\ref{sec:time-scale}, we first present a general time-scale separation results, which will be used in Section~\ref{sec:IC} to provide a precise description of the early stage of the dynamics, when $\pi^0_t$ reaches $\Delta$. Then, in Section~\ref{sec:poisson}, we will provide results on the well-posedness of Equation~\eqref{eq:Poisson}. We will then be in position to prove~Theorem~\ref{thm:soft-kill-1} in Section~\ref{sec:ProofT1}.

\subsection{A time-scale separation result}\label{sec:time-scale}

In the selection/mutation mechanism, there exists a cascade of time-scales. The largest time scale is the mutation one, which is independent of $r\in\N$. Besides, we may define an equivalence relation on $D$ by $x\sim y$ if and only if $0<\alpha_{x,y,\infty}<\infty$. Then, each equivalence class associated with this equivalence relation gives a different time scale, as, if $\alpha_{x,y,\infty}=\infty$, then all particles in $y$ will jump away from $y$ to go to a site with a lower death rate before any particle in $x$ have time to jump or move, with probability going to $1$ as $r$ goes to infinity. Lemma~\ref{lem:different-time-scale} below describes this phenomenon, and is a cornerstone in the proofs of the results under Assumption~\ref{assu:assu1}. In particular, if at a given time, the empirical process is not a Dirac mass, it will give us information on the time taken by the process to reach the set of Dirac measures and on how the particles will be redistributed to form a Dirac mass.

\begin{lem}\label{lem:different-time-scale}
Let $a_r>0$, $b_r \geqslant 0$,  $L_r = a_r L_r^1 + b_r L_r^2$ be a jump Markov process generator on the state space $\mathcal M_n^1(D)$, such that $L^1_r$ is a selection dynamics with partial rates: 
\[
\forall f \in \ell^\infty(\mathcal M^1_n(D)), \, L_r^{1}f(\eta) = \sum_{x\in \Theta,y\in D} \frac{n^2}{n-1} \eta(x)\gamma_r(x)\eta(y) \left( f\left(\eta +\frac{\delta_y - \delta_x}{n}\right) - f(\eta) \right),
\]
for some $\Theta\subset D$ and $\gamma_r\geqslant 1$, and such that the coefficients of $L^2_r$ are bounded from above in $r$:  
\[
\forall f \in \ell^\infty(\mathcal M^1_n(D)), \, L^2_rf(\eta) = \sum_{\xi\in \mathcal M^1_n(D)}\beta^2_r(\eta,\xi) \po f(\xi)-f(\eta) \pf, \qquad\sum_{\xi\in \mathcal M^1_n(D)}\beta^2_r(\eta,\xi) \leqslant \bar \mu,\] for some $\bar \mu >0$ and all $\eta \in \mathcal M^1_n(D)$.
Let us assume that
\[
\lim_{r\rightarrow \infty} \frac{b_r}{a_r}  = 0,
\]
and let us introduce $c_r>0$ such that 
\[
\lim_{r\rightarrow \infty} \frac{c_r}{a_r} = \lim_{r \rightarrow \infty} \frac{b_r}{c_r} = 0.
\]
Write
\[
\mathbb M^1= \left\{ \xi\in \mathcal M_n^1(D), L_r^1 f(\xi) = 0,\, \forall f \in \ell^\infty(\mathcal M^1_n(D)),\, \forall r\in \N \right\},
\]
and
\[
\tau_1^r = \inf\left\{t\geqslant 0, \tilde\pi^r_t \in \mathbb M^1\right\},
\]
where $\tilde \pi^r$ is a Markov process with generator $L_r$. Denote by $S_1$ the first jump of the process coming from the $b_r L_r^2$ part of the generator.
Then the following holds:
\begin{equation}\label{eq:time-scale}
\lim_{r\rightarrow\infty} \inf_{\xi \in\M^1_n(D)}\mathbb P_{\xi}\po \tau_1^r < c_r^{-1} < S_1 \pf = 1.
\end{equation}
Moreover, if $L^2_r = 0$, then there exist $C,c>0$ such that
\begin{equation}\label{eq:bound-death-time}
\sup_{r\in \N}\sup_{\xi\in \mathcal M^1_n(D),} \mathbb P_{\xi}\po \tau_1^r \geqslant t \pf \leqslant  Ce^{-c a_rt}.
\end{equation}
\end{lem}

This lemma will be applied to the Fleming--Viot generator~\eqref{eq:decompos_L}, either with $\Theta=D$ or with $\Theta$ a subset of $D$ such that the killing rates on $\Theta$ are much greater then the ones on $D\setminus \Theta$, in the sense that for all $x\in \Theta$, $y\notin \Theta$, $\alpha_{x,y,\infty} = 0$. In the latter case, the set $\mathbb M^1$ contains the Dirac measures, as well as the empirical measures of $n$ particles with support in $D\setminus \Theta$.

\begin{proof}
Let $\bar \pi$ be a Markov process with generator $L^1_r$, defined using the same jump events as $\tilde \pi$, and
\[
\bar \tau_1^r = \inf\left\{t\geqslant 0, \bar\pi^r_t \in\mathbb M^1\right\}.
\]
On the event $\tau_1^r<S_1$, we have $\tau_1^r = a_r^{-1}\bar \tau_1^r$. Let us show that there exist $C,c>0$ such that
\begin{equation}\label{eq:bound-death-time-normalized}
\sup_{r\in \N}\sup_{\xi\in \mathcal M^1_n(D),} \mathbb P_{\xi}\po \bar\tau_1^r \geqslant t \pf \leqslant Ce^{-c t}.
\end{equation}
Fix $\xi\in \mathcal M^1_n(D)$, write $D_0 = {\rm supp}(\xi)$ and
\[
\underline \gamma_r = \min \left\{ \gamma_r(x), x\in D_0 \cap \Theta \right\}.
\]
Since $\gamma_r(x)\geqslant 1$ for all $x\in D$, $\underline \gamma_r \geqslant 1$ for all $r\in\N$. Let $\underline x_r \in D$ be such that $\underline \gamma_r= \gamma_r( \underline x_r)$ if $D_0\subset \Theta$, and let $\underline x_r \in D_0\setminus \Theta$ otherwise. Let $\X$ be a process whose empirical measure is given by $\bar \pi$, and let $i_0\in \llbracket 1,n\rrbracket$ be such that $X^{i_0}_0 = \underline x_r$. Denote
\[
I = \left\{i,\, X^i_0 \in D_0\setminus\Theta \right\}\cup \left\{i_0\right\},\quad n_0 = \left|I\setminus \left\{i_0\right\}\right|,
\]
and write:
\[
\mathcal A_r = \left\{ \text{On }[0,1/\underline \gamma_r], X^{i},\, i\in I, \text{ do not die, and all other die once and branch on } X^{i_0} \right\}.
\]
We have that $\mathcal A_r \subset \left\{ \bar \tau_1^r < 1\right\}$. Since the number of death events for $(X^i_s)_{0 \leqslant s \leqslant 1/\underline \gamma_r}$  is upper (resp. lower) bounded by a Poisson random variable with parameter $\underline \gamma_r$ if $i\in I$ (resp. $i\notin I$), we have that:
\[
\mathbb P_{\xi}\left(  \mathcal A_r\right) \geqslant \frac{e^{-n}\po 1-e^{-1}\pf^{n-n_0}}{n^{n-n_0}},
\]
and hence, we get that 
\[
\inf_{r\in\N }\inf_{\xi\in \mathcal M^1_n(D)} \mathbb P_{\xi}\po \bar \tau_1^r <1 \pf \geqslant  1-\omega,\]
for some $\omega \in [0,1)$. Using the Markov property at time $t-1$, this implies that for all $r\in \N$, $\xi\in \mathcal M^1_n(D)$, and $t\geqslant 1$:
\[
\mathbb P_{\xi}\po \bar \tau_1^r \geqslant t \pf  \leqslant \omega \mathbb P_{\xi}\po \bar \tau_1^r \geqslant t-1 \pf, 
\]
and this implies 
\[\sup_{r\in\N}\sup_{\xi\in \mathcal M^1_n(D)} \mathbb P_{\xi}\po \bar \tau_1^r > t \pf \leqslant \omega^{\lfloor t \rfloor} \leqslant Ce^{-c t},
\]
with $c = -\ln(\omega)>0$ and $C=\omega^{-1}$.

Writing $\beta^2_r(\eta)=\sum_{\xi\in\mathcal M^1_n(D)}\beta^2_r(\eta,\xi)$, $S_1$ is defined as:
\[
S_1 = \inf\left\{ t\geqslant 0, E_1 \leqslant \int_0^t b_r\beta^2_r(\tilde \pi_s^r) \dd s \right\} \geqslant \inf\left\{ t\geqslant 0, E_1 \leqslant b_r\bar \mu t \right\} = b_r^{-1}\tilde S_1,
\]
where $E_1$ is an exponential random variable with parameter $1$. Hence, $S_1$ is bounded from below by $b_r^{-1}\tilde S_1$, where $\tilde S_1$ is an exponential random variable with parameter $\bar\mu$. Finally, using the bound~\eqref{eq:bound-death-time-normalized} we have:
\begin{align*}
    \mathbb P\po \tau_1^r < c_r^{-1} < S_1 \pf &= \mathbb P\po a_r^{-1}\bar \tau_1^r < c_r^{-1} < S_1 \pf \\ &\geqslant 1 - \mathbb P\po \bar \tau_1^r \geqslant \frac{a_r}{c_r} \pf - \mathbb P\po   S_1 < c_r^{-1}\pf \\ &\geqslant 1 - \mathbb P\po \bar \tau_1^r \geqslant \frac{a_r}{c_r} \pf - \mathbb P\po  \tilde S_1 < \frac{b_r}{c_r}\pf \underset{r\rightarrow\infty}{\rightarrow} 1,
\end{align*}
and hence the limit~\eqref{eq:time-scale}. In the case where $L^2_r=0$, Equation~\eqref{eq:bound-death-time-normalized} yields:
\[
\mathbb P_{\xi}\po \tau_1^r \geqslant t \pf = \mathbb P_{\xi}\po \bar\tau_1^r \geqslant a_rt \pf\leqslant Ce^{-ca_rt},
\]
which concludes the proof.
\end{proof}

\subsection{The initial condition}~\label{sec:IC}

Now, we aim to address the issue of the initial condition. Given some initial condition $\eta \in \mathcal M^1_n(D)$, write $D_0 = {\rm supp}(\eta)$, where
\[
{\rm supp}(\eta) = \left\{ x\in D, \eta(x)>0 \right\},
\] 
and 
\begin{equation}\label{eq:minimal-death-rates}
\Lambda = \left\{ x\in D_0,\, \min_{y\in D_0} \alpha_{x,y,\infty} >0 \right\}.
\end{equation}
This set $\Lambda$ correspond to the subset of $D_0$ where, asymptotically in $r$, the order of magnitude of the killing rate is minimal. Thanks to Lemma~\ref{lem:different-time-scale}, we will be able to show that in the limit~\eqref{eq:context}, all particles from $D_0\setminus\Lambda$ will be redistributed instantaneously (ie on a time scale faster than $1/\inf_{D_0} \lambda_r$) in $\Lambda$, with no movement from the particles in $\Lambda$. Moreover, the law of the configuration obtained in $\Lambda$ after this instantaneous redistribution is given by a Polya's urn problem, with $|\Lambda|$ possible colors and initial condition $n\eta|_{\Lambda}$, after $n\eta(D_0\setminus\Lambda)$ draws.
On the time scale $1/\inf_{D_0} \lambda_r$, this measure will then form a Dirac mass according to the probability given by commitor functions, that we are now going to define.

Recall that $\pi^0$ denotes a Markov process with generator $L_{s,0}$, and the definition~\eqref{eq:tau_urne} of $\tau$. Let us apply Lemma~\ref{lem:different-time-scale}, with $L^2_r=0$, $a_r=1$, $\Theta = {\rm supp}(\pi^0_0)$,  $\gamma_r = \alpha_r$, so that $L^1_r=L_{s,0}$  and $\tau^1_r=\tau$.
Equation~\eqref{eq:bound-death-time} then yields that $\tau <\infty$ almost surely. Denote for $\eta\in\mathcal M^1_n(D)$:
\begin{equation}\label{eq:commitor}
\psi_x^r(\eta) = \mathbb P_{\eta}\po \pi^0_{\tau} = \delta_x \pf.
\end{equation}
The functions $\psi_x^r$ are the so called committor functions, and \cite[Theorem 3.3.1.]{Norris} yields that they satisfy:
\begin{equation}\label{eq:commitore}
  \psi_x^r(\delta_x) = 1,\, \psi_x^r(\delta_y) = 0\, \forall y\neq x,\, L_{s,0}\psi_x^r(\xi) = 0,\, \forall \xi\in\mathcal M^1_n(D)\setminus \Delta.  
\end{equation}
Remark that for all $\xi\in \Delta$, $f\in \ell^{\infty}\po \mathcal M^1_n(D)\pf$, $L_{s,0}f(\xi) = 0$, so that we actually have $L_{s,0}\psi_x^r(\xi) = 0$, for all $\xi\in\mathcal M^1_n(D)$. The next lemma proves that the committor functions are the unique solution to~\eqref{eq:commitore} in~$\ell^\infty \po \mathcal M_n^1(D) \pf$.

\begin{lem}\label{lem:commitor}
Let $\Theta\subset D$ and $\gamma: \Theta \to [1,+\infty)$.
Let $L^1$ be defined by: $\forall  f \in \ell^{\infty}\po \mathcal M^1_n(\Theta) \pf$,
\begin{equation*}
L^{1}f(\eta) = \sum_{x,y\in \Theta} \frac{n^2}{n-1} \eta(x)\gamma(x)\eta(y) \left( f\left(\eta +\frac{\delta_y - \delta_x}{n}\right) - f(\eta) \right).
\end{equation*}
For all $x\in \Theta$, if $f\in \ell^{\infty}\po \mathcal M^1_n(\Theta) \pf$ satisfies
\begin{equation}\label{eq:Laplace}
f(\delta_x) = 1,\, f(\delta_y) = 0\, \forall y\neq x,\, L^1f(\xi) = 0,\, \forall \xi\in\mathcal M^1_n(\Theta)\setminus\Delta,
\end{equation}
then for all $\xi \in\mathcal M_n^1(\Theta)$
\[
f(\xi)= \mathbb P_{\xi} \po \pi^1_{\tau_1} = \delta_x \pf,
\]
where $\tau_1 = \inf\left\{ t\geqslant 0, \pi^1_t \in \Delta \right\}$ and $\pi^1$ is a Markov process with generator $L^1$ and values in $\mathcal M^1_n(\Theta)$. 

In particular, $\psi_x^r$ defined by~\eqref{eq:commitor} are the unique solution to~\eqref{eq:commitore} in $\ell^{\infty}\po \mathcal M^1_n(D) \pf$.
\end{lem}

\begin{proof}
Let $f$ be a solution to Equation~\eqref{eq:Laplace}. Let $\xi\in \mathcal M^1_n(\Theta)$ and $\pi^1$ be a Markov process with generator $L^1$ and initial condition $\xi$. The map $t\mapsto f(\pi^1_{\tau_1\wedge t})$ is a martingale, and hence for all $t\geqslant 0$, 
\begin{equation}\label{eq:martingal}
\E_{\xi}\po f(\pi^1_{\tau_1\wedge t}) \pf = f(\xi).
\end{equation}
Equation~\eqref{eq:bound-death-time} from Lemma~\ref{lem:different-time-scale} yields that $\tau_1<\infty$ almost surely, and since $f$ is bounded, letting $t$ go to infinity in~\eqref{eq:martingal} yields that $f(\xi) = \mathbb P_{\xi} \po \pi^1_{\tau_1} = \delta_x \pf$.
\end{proof}

We are now able to prove

\begin{lem}\label{lem:initial-condi}
For all $\eta \in \mathcal M^1_n(D)$, there exists $\eta_{\infty} \in \mathcal M^1(D)$ such that if $\pi^0$ is a Markov process with generator $L_{s,0}$ and initial condition $\eta$, and $Y_0$ is a random variable with law $\eta_{\infty}$, then
\[
\lim_{r\rightarrow\infty}\|\mathcal Law(\pi_{\tau}^0)- \mathcal Law(\delta_{Y_0})\|_{TV}=0,
\]
where, we recall, $L_{s,0}$ defined by~\eqref{eq:Ls0} implicitly depends on $r$, and $\tau$ is defined by~\eqref{eq:tau_urne}.
\end{lem}

In the statement of Theorem~\ref{thm:soft-kill-1}, $\eta_{\infty}$ is the law of the initial condition $Y_0$.

\begin{proof}
Fix $\eta\in\mathcal M^1_n(D)$ and write $\bar \eta_{r} = \mathcal Law(\pi_{\tau}^0)\in\mathcal M^1\po \Delta \pf$. Denote $D_0(\eta) = {\rm supp}(\eta)$ and
\[
\underline \alpha_r = \inf_{x\in D_0}\alpha_r(x),
\] 
where $\alpha_r(x)$ is defined in~\eqref{eq:normalised-coef}. Recall the definition of $\Lambda$ from~\eqref{eq:minimal-death-rates}. The proof is divided into two parts. If $\Lambda = D_0$, then we can find the limit of $\bar\eta_{r}$ when $r\rightarrow\infty$ using the fact that $\bar\eta_{r}(\delta_x) = \psi_x^r(\eta)$ and that $(\psi_x^r(\eta))_{\eta\in\mathcal M^1_n(\Lambda)}$ is the unique solution to a Dirichlet equation on a finite set (see Lemma~\ref{lem:commitor}). If $\Lambda \neq D_0$, then using Lemma~\ref{lem:different-time-scale}, we show that all particles from $D_0\setminus\Lambda$ will first be distributed on $\Lambda$, without any particles from $\Lambda$ moving, and then we can apply the argument from the case $D_0 = \Lambda$.

Let us suppose that $\Lambda = D_0$. Write $L^{D_0,r}_{s,0}$ the re-normalised generator of the selection dynamics 
\[
L_{s,0}^{D_0,r}f(\xi) = \sum_{x,y\in D_0} \frac{\alpha_r(x)}{\underline \alpha_r} \xi(x)\xi(y) \left( f\left(\xi +\frac{\delta_y - \delta_x}{n}\right) - f(\xi) \right),
\]
defined for $f:\mathcal M^1_n(D_0)\to\R$. Denote also 
\[
\alpha_{D_0,\infty}(x) = \lim_{r\rightarrow\infty} \frac{\alpha_r(x)}{\underline \alpha_r} = \frac{1}{\min_{y \in D_0} \alpha_{x,y,\infty}},
\]
which is well defined and lies in $[1,\infty)$ since $D_0=\Lambda$, as well as
\[
L_{s,0}^{D_0,\infty}f(\xi) = \sum_{x,y\in D_0} \alpha_{D_0,\infty}(x)\xi(x)\xi(y) \left( f\left(\xi +\frac{\delta_y - \delta_x}{n}\right) - f(\xi) \right).
\]
The operator $L_{s,0}^{D_0,\infty}$ is the generator of a selection dynamics with death rates $\alpha_{D_0,\infty}$.
We have that
\begin{equation*}
\bar\eta_{r}(\delta_x) = \psi_x^{r}(\eta),
\end{equation*}
and for all $x\in D_0$ and $\xi \in \mathcal M^1_n(D_0)$,
\begin{equation}\label{eq:finite_set_poisson}
 L_{s,0}^{D_0,r}\psi^{r}_x(\xi) = 0. 
\end{equation}
Since $D_0\times \mathcal M^1_n(D_0)$ is a finite set, we may consider $(\psi_x^{\infty}(\xi))_{x\in D_0,\xi\in \mathcal M^1_n(D_0)}$ a sub-sequential limit of $(\psi_x^{r}(\xi))_{x\in D_0,\xi\in \mathcal M^1_n(D_0)}$. Passing to the limit $r \to \infty$ in Equation~\eqref{eq:finite_set_poisson}, we get that for all $x\in D_0$ and $\xi \in \mathcal M^1_n(D_0)$:
\[
L_{s,0}^{D_0,\infty}\psi_x^{\infty}(\xi) = 0,
\]
as well as $\psi_x^{\infty}(\delta_y) = \mathbbm 1_{x\neq y}$. Lemma~\ref{lem:commitor} then implies that
\[
\psi_x^{\infty}(\xi) =\mathbb P_{\xi}\po \pi^{0,\infty}_{\tau} = \delta_x \pf,
\]
where $\pi^{0,\infty}$ is a Markov process with generator $L_{s,0}^{D_0,\infty}$. Hence there is a unique sub-sequential limit, and $\psi_x^{r}(\xi)$ converge for all $\xi \in \mathcal M^1_n(D_0)$ and $x\in D_0$, and this is particularly true for $\xi = \eta$.

We now suppose that $\Lambda \neq D_0$. In this case, for all $x\in D_0\setminus\Lambda$, $\alpha_{D_0,\infty}=\infty$ and $L^{D_0,\infty}$ is ill-defined. To resort to the previous case, write 
\[S_1^{\Lambda}=\inf\left\{ t \geqslant 0, \text{ one particle starting from }\Lambda\text{ dies} \right\},
\]
and
\[\tau_{\mathcal M^1_n(\Lambda)} = \inf\left\{t\geqslant 0, \pi^0_t \in\mathcal M^1_n(\Lambda) \right\}.
\]
Using the notation of Lemma~\ref{lem:different-time-scale}, write 
\[\Theta = D_0\setminus\Lambda,\quad a_r = \inf\left\{ \alpha_r(x)/\underline\alpha_r, x\in D_0 \setminus \Lambda \right\},\quad  b_r = 1,\quad L_r^2 = L_{s,0}^{D_0,r}.\]
In this case, we have $\mathbb M^1 = \Delta \cup \mathcal M^1_n(\Lambda)$, and because $\Lambda\neq D_0$, $a_r\underset{r\rightarrow\infty}{\rightarrow} \infty$, and Lemma~\ref{lem:different-time-scale} (with $c_r = \sqrt{a_r}$) yields that:
\[
\lim_{r\rightarrow\infty}\mathbb P_{\eta}\po S_1^{\Lambda} \leqslant  \tau_{\mathcal M^1_n(\Lambda)} \pf = 0.
\]
Now, using the strong Markov property at time $\tau_{\mathcal M^1_n(\Lambda)}$, write:
\[
\mathbb P_{\eta}\po \pi^0_{\tau} = \delta_x \pf = \sum_{\xi\in \mathcal M^1_n(\Lambda)} \mathbb P_{\eta}\po S_1^{\Lambda} >  \tau_{\mathcal M^1_n(\Lambda)} \pf \mathbb P_{\eta}\po \pi^0_{\tau_{\mathcal M^1_n(\Lambda)}} = \xi \middle| S_1^{\Lambda} >  \tau_{\mathcal M^1_n(\Lambda)} \pf \mathbb P_{\xi}\po \pi^0_{\tau} = \delta_x \pf + R_r, 
\]
where a rest satisfying 
\[R_r \leqslant \mathbb P_{\eta}\po S_1^{\Lambda} \leqslant  \tau_{\mathcal M^1_n(\Lambda)} \pf \underset{r\rightarrow\infty}{\rightarrow}0.\]
From the case $D_0=\Lambda$ we get that for all $\xi \in \mathcal M^1_n(\Lambda)$, $\mathbb P_{\xi}\po \pi^0_{\tau} = \delta_x \pf$ admits a limit as $r$ goes to infinity. In addition, 
\[
\mathbb P_{\eta}\po \pi^0_{\tau_{\mathcal M^1_n(\Lambda)}} = \pi \middle| S_1^{\Lambda} >  \tau_{\mathcal M^1_n(\Lambda)} \pf
\] 
corresponds to the law of a Polya urn Markov chain with with color indexed by $\Lambda$, initial condition $n\eta|_{\Lambda}$, after $|D_0\setminus\Lambda|$ draws, and is hence independent of $r$. This concludes the convergence of $\mathbb P_{\eta}\po \pi^0_{\tau} = \delta_x \pf$ towards some $\bar \eta_{\infty}\in \mathcal M^1(\Delta)$. This concludes the proof since if $Y_0\in D$ is such that $\mathcal Law(\delta_{Y_0}) = \bar \eta_{\infty}$, then $\eta_{\infty}$ is the law of $Y_0$.
\end{proof}

\subsection{The Poisson equation~\texorpdfstring{\eqref{eq:Poisson}}{(21)}}\label{sec:poisson}

Let us now consider the Poisson equation~\eqref{eq:Poisson}. Before showing existence and uniqueness of a solution to this equation, let us describe the kernel of $L_{s,0}^*$ and $L_{s,0}$. Although this is not strictly necessary for the proof of Theorem~\ref{thm:soft-kill-1}, this makes the functional setting clearer. To proceed, let us recall that $L_{s,0}$ is defined as an operator of $\ell^{\infty}\po \mathcal M^1_n(D)\pf$, and $L_{s,0}^*$ as an operator of $\ell^{1}\po \mathcal M^1_n(D)\pf$. Define the weak-$\ast$ topology on $\ell^{\infty}\po \mathcal M^1_n(D)\pf$ by the fact that a sequence of elements $(w_k)_{k \geqslant 1}$ weak-$\ast$ converges to $w$ if $\langle v,w_k\rangle \to \langle v,w\rangle$ for any $v \in \ell^1(\mathcal M^1_n(D))$. Then $\ell^{1}\po \mathcal M^1_n(D)\pf$ equipped with the (strong) topology induced by the norm $\|\cdot\|_1$, and $\ell^{\infty}(\mathcal M^1_n(D))$ equipped with the weak-$\ast$ topology, are each other's topological dual. Let us now write 
\[
(\mathrm{Ker}L_{s,0})^{\perp} = \left\{v \in \ell^1\po \mathcal M^1_n(D) \pf,\, \langle v,w \rangle = 0,\, \forall w\in \mathrm{Ker} L_{s,0} \right\},
\]
recall the definition of the commitor functions from~\eqref{eq:commitor}, and denote by $\overline A$ the closure of a set $A$ for the topology at hand.

\begin{lem}\label{lem:topology}
We have
\[
(\mathrm{Ker}L_{s,0})^{\perp} = \{v \in \ell^1\po \mathcal M^1_n(D) \pf,\,\langle v, \psi_x^r\rangle = 0,\, \forall x\in D\},
\]
as well as
\[
 \mathrm{Ker} L_{s,0}^* =  \overline{\mathrm{Span}(\mathbbm{1}_x, x \in D)} .
\]
In the case where $D$ is infinite, the inclusion
\[
\mathrm{Span}(\psi_x, x \in D) \subset \mathrm{Ker} L_{s,0}
\]
is strict, and the kernel of $L_{s,0}$ is given by 
\[
\mathrm{Ker} L_{s,0} = \overline{\mathrm{Span}(\psi_x, x \in D)}^{w*},
\]
where $\overline{A}^{w*}$ denotes the closure for the weak-$\ast$ topology of a set $A\subset \ell^{\infty}(\mathcal M^1_n(D))$.
\end{lem}
\begin{proof}
For all $x\in D$, Lemma~\ref{lem:commitor} yields that $\psi_x^r\in \mathrm{Ker} L_{s,0}$, so that:
\[
(\mathrm{Ker} L_{s,0})^\perp \subset \{v \in \ell^1\po \mathcal M^1_n(D) \pf,\,\langle v, \psi_x^r\rangle = 0,\, \forall x\in D\} .
\]
Let us now prove the reverse inclusion. Let $v\in \ell^1\po \mathcal M^1_n(D) \pf$ such that $\langle v, \psi_x^r\rangle = 0$ for all $x\in D$, and $w\in \mathrm{Ker}L_{s,0}$ (and thus $w\in \ell^{\infty}\po \mathcal M^1_n(D)\pf$. The map $t\mapsto w(\pi^0_{t\wedge \tau})$ is a martingale, and hence for all $t\geqslant 0$, 
\begin{equation*}
w(\xi) = \E_{\xi}\po w(\pi^0_{\tau\wedge t}) \pf.
\end{equation*}
Since $w$ is bounded, and $\tau<\infty$ almost surely, by letting $t$ go to infinity,  the dominated convergence theorem yields that for all $\xi\in \mathcal M_n^1(D)$:
\begin{equation}\label{eq:noyau-L0}
w(\xi) = \sum_{x\in D}w(\delta_x) \psi_x(\xi),
\end{equation}
and the sum is actually finite because $\psi_x(\xi)=0$ if $x \notin {\rm supp}(\xi)$.  We have:
\[
\sum_{\xi\in\mathcal M^1_n(D)} \sum_{x\in D}|w(\delta_x)v(\xi)\psi_x(\xi)| = \sum_{\xi\in\mathcal M^1_n(D)}| v(\xi)| \sum_{x\in D} |w(\delta_x)|\psi_x^r(\xi) \leqslant \|v\|_1 \|w\|_\infty < \infty. 
\]
Hence, the following computation holds:
\begin{equation*}
\langle v, w\rangle = \sum_{\xi\in \mathcal M_n^1(D)}v(\xi)\sum_{x\in D} w(\delta_x) \psi_x(\xi) = \sum_{x\in D} w(\delta_x) \sum_{\xi\in \mathcal M_n^1(D)}v(\xi) \psi_x(\xi) =  \sum_{x\in D} w(\delta_x)\langle v, \psi_x^r\rangle = 0,
\end{equation*}
which yields that $v\in (\mathrm{Ker}L_{s,0})^{\perp}$, and concludes the first point.

We already saw in the introduction of Section~\ref{sec:general-rate} that $\mathrm{Span}(\mathbbm{1}_x, x \in D) \subset \mathrm{Ker} L_{s,0}^*$. Since $L_{s,0}^*$ is a bounded operator, $\mathrm{Ker} L_{s,0}^*$ is closed and  \[\overline{\mathrm{Span}(\mathbbm{1}_x, x \in D)} \subset \mathrm{Ker} L_{s,0}^*.\]
Let us now prove the reverse inclusion. For $v\in \mathrm{Ker} L_{s,0}^*$, the map \[t\mapsto \sum_{\xi\in\mathcal M^1_n(D)}v(\xi)\mathbb P_{\xi}\po \pi_t^0 = \eta \pf\] is constant for all $\eta\in \mathcal M^1_n(D)$. Write for all $\eta \in \mathcal M^1_n(D)$:
\begin{align*}
    \mathbb P_{\xi} \po \pi^0_s = \eta \pf - \mathbb P_{\xi} \po \pi^0_{\tau} = \eta \pf & = \mathbb P_{\xi} \po \pi^0_s = \eta,\tau>s \pf + \mathbb P_{\xi} \po \pi^0_s = \eta,\tau<s \pf - \mathbb P_{\xi} \po \pi^0_{\tau} = \eta \pf \\ &= \mathbb P_{\xi} \po \pi^0_s = \eta,\tau>s \pf - \mathbb P_{\xi} \po \pi^0_{\tau} = \eta,\tau>s \pf,
\end{align*}
which yields
\begin{equation}\label{eq:bound-coupling}
\sum_{\eta\in\mathcal M^1_n(D)} \left| \mathbb P_{\xi} \po \pi^0_s = \eta \pf - \mathbb P_{\xi} \po \pi^0_{\tau} = \eta \pf \right| \leqslant 2\mathbb P_{\xi}\po \tau >s \pf. 
\end{equation}
Using the notations of Lemma~\ref{lem:different-time-scale}, write:
\[
\Theta = D,\quad a_r = 1,\quad \gamma_r = \alpha_r,\quad L^2_r = 0.
\]
Lemma~\ref{lem:different-time-scale} then yields that there exist $C,c>0$ such that for all $s\geqslant 0$
\begin{equation}\label{eq:bound-tau-dirac}
    \sup_{r\in \N}\sup_{\xi\in \mathcal M^1_n(D),} \mathbb P_{\xi}\po \tau >s \pf \leqslant Ce^{-cs}.
\end{equation}
The bound~\eqref{eq:bound-coupling} and~\eqref{eq:bound-tau-dirac} yields that we may let $t$ go to infinity and get:
\[
v = \sum_{x\in D}v(\delta_x) \mathbbm 1_x.
\]
Since this sum converges in $\ell^{1}\po \mathcal M^1_n(D)\pf$ as soon as $v\in \ell^{1}\po \mathcal M^1_n(D)\pf$, $v\in \overline{\mathrm{Span}(\mathbbm{1}_x, x \in D)}$, which concludes the proof of the second point.
 
Define for all $\xi\in \mathcal M^1_n(D)$, 
\[
f(\xi) = \sum_{x\in D} \psi_x(\xi).
\]
In this sum, there are at most $n$ non-vanishing terms, and thus $f\in \ell^{\infty}\po \mathcal M^1_n(D) \pf$. It is easily seen that $L_{s,0}f=0$. If $g=\sum_{x\in D_0} a_x\psi_x(\xi)$ for some finite subset $D_0\subset D$ and $a_x\in \R$, because $D$ is infinite, $D_0\neq D$ and we have that $\|f-g\|_{\infty} \geqslant 1$, yielding that 
\[
f\in {\rm Ker} L_{s,0}\setminus \overline{\mathrm{Span}(\psi_x, x \in D)}.
\]

Finally, for all $w\in \mathrm{Ker}L_{s,0}$, we have shown that $w$ satisfies~\eqref{eq:noyau-L0}, and the sum in the right-hand side of~\eqref{eq:noyau-L0} converges for the weak-$\ast$ topology, which proves the last point.
\end{proof}

We are now ready to solve the Poisson equation~\eqref{eq:Poisson}, relying on the next lemma. For any $v \in \ell^1(\mathcal M^1_n(D))$ and $s \geqslant 0$, let us define $v e^{sL_{s,0}}$ by:
\begin{equation}\label{eq:semi-group}
\forall \eta \in \mathcal M^1_n(D),\qquad    v e^{sL_{s,0}}(\eta) = \sum_{\xi \in \mathcal M^1_n(D)} v(\xi) \mathbb{P}_\xi(\pi^0_s=\eta),
\end{equation}
which defines an element of $\ell^1(\mathcal M^1_n(D))$.

\begin{lem}\label{lem:poisson-eq}
For all $v\in (\mathrm{Ker} L_{s,0})^\perp (\subset \ell^1(\mathcal M^1_n(D)))$, there exists a unique $w\in (\mathrm{Ker} L_{s,0})^\perp$ such that 
\begin{equation}\label{eq:PPoisson}
L_{s,0}^*w = v,
\end{equation}
given by:
\[
w = -\int_0^{\infty} ve^{s L_{s,0}} \dd s.
\]
Moreover, there exists $C>0$ independent of $r$ and $v$ such that
\[
\|w\|_{1} \leqslant C\|v\|_{1}.
\]

\end{lem}


\begin{proof}
Let us start with the existence. Write, for $\eta \in \mathcal M^1_n(D)$ 
\[
w(\eta) = -\po\int_0^{\infty} ve^{s L_{s,0}} \dd s \pf(\eta) = -\int_0^{\infty}\sum_{\xi \in \mathcal M^1_n(D)} v(\xi) \mathbb P_{\xi} \po \pi^0_s = \eta \pf \dd s.
\]
We first show that $w$ is well-defined. Recall the stopping time $\tau$ defined in~\eqref{eq:tau_urne}. Since $v\in (\mathrm{Ker} L_{s,0})^\perp $, Lemma~\ref{lem:topology} yields that $\langle v,\psi_x^r \rangle = 0$ for all $x\in D$, and we have
\[
\sum_{\xi \in \mathcal M^1_n(D)} v(\xi) \mathbb P_{\xi} \po \pi^0_{\tau} = \delta_x \pf = \sum_{\xi \in \mathcal M^1_n(D)} v(\xi)\psi_x^r(\xi) = 0.
\]
Since, if $\eta \notin \Delta$, $\mathbb P_{\xi}\po \pi^0_{\tau} = \eta \pf = 0$, we have that
\[
\sum_{\xi \in \mathcal M^1_n(D)} v(\xi) \mathbb P_{\xi} \po \pi^0_s = \eta \pf = \sum_{\xi \in \mathcal M^1_n(D)} v(\xi) \po  \mathbb P_{\xi} \po \pi^0_s = \eta \pf - \mathbb P_{\xi} \po \pi^0_{\tau} = \eta \pf \pf.
\]
The bounds~\eqref{eq:bound-coupling} and~\eqref{eq:bound-tau-dirac} imply that $w$ is well-defined, because $v\in \ell^1(\mathcal M^1_n(D))$, and we have
\[
\|w\|_{1} \leqslant 2C\int_0^{\infty} e^{-cs} \dd s \|v\|_{1} = \frac{2C}{c}\|v\|_{1}.
\]

The Kolmogorov equations yield that $L_{s,0}^*\mathbb P_{\xi} \po \pi^0_s = \cdot \pf = \partial_s \mathbb P_{\xi} \po \pi^0_s = \cdot \pf$, and hence for all $\eta\in \mathcal M^1_n(D)$:
\begin{align*}
L_{s,0}^*w(\eta) &= -\int_0^{\infty}\sum_{\xi \in \mathcal M^1_n(D)} v(\xi) L_{s,0}^*\mathbb P_{\xi} \po \pi^0_s = \cdot \pf \po \eta\pf \dd s \\ &= -\int_0^{\infty}\sum_{\xi \in \mathcal M^1_n(D)} v(\xi) \partial_s \mathbb P_{\xi} \po \pi^0_s = \eta \pf \dd s \\ &= -\int_0^{\infty}\partial_s\po\sum_{\xi \in \mathcal M^1_n(D)} v(\xi) \mathbb P_{\xi} \po \pi^0_s = \eta \pf \pf \dd s.
\end{align*}
We already saw that 
\[
\lim_{s\rightarrow\infty} \sum_{\xi \in \mathcal M^1_n(D)} v(\xi) \mathbb P_{\xi} \po \pi^0_s = \eta \pf = 0,
\]
and hence $L_{s,0}^*w(\eta) = v(\eta)$.

Let $f\in \mathrm{Ker} L_{s,0}$. We have that
\[
\langle w,f\rangle = - \int_0^{\infty} \sum_{\xi,\eta\in\mathcal M_n^1(D)}v(\xi)\mathbb P_{\xi}\po \pi_s^0= \eta \pf f(\eta) \dd s = - \int_0^{\infty} \sum_{\xi\in\mathcal M_n^1(D)}v(\xi)\po e^{sL_{s,0}}f \pf(\xi)\dd s = 0,
\]
so that $w\in (\mathrm{Ker} L_{s,0})^\perp$, which concludes the existence.

To prove uniqueness, fix $w\in \mathrm{Ker} L_{s,0}^* \cap  (\mathrm{Ker} L_{s,0})^\perp $. The fact that $L^*_{s,0}w= 0$ yields that for all $\eta \in\mathcal M_n^1(D)$, $t\geqslant 0$:
\[
w(\eta) = \sum_{\xi\in\mathcal M_n^1(D)} w(\xi)\mathbb P_{\xi}\po \pi_t^0 = \eta \pf.
\]
Letting $t\rightarrow\infty$, because $w\in \ell^1\po \mathcal M^1_n(D)\pf$, Inequalities~\eqref{eq:bound-coupling}, and~\eqref{eq:bound-tau-dirac} then imply that
\[
w(\eta) = \sum_{\xi\in\mathcal M_n^1(D)} w(\xi)\mathbb P_{\xi}\po \pi_{\tau}^0 = \eta \pf\, \forall \eta\in \mathcal M_n^1(D).
\]
In particular, $w(\eta)=0$ for all $\eta\in \mathcal M_n^1(D)\setminus \Delta$,
and thus we may write $w(\eta) = \sum_{x\in D} w(\delta_x)\mathbbm 1_x(\eta)$, where $\mathbbm 1_x$ was defined in~\eqref{eq:dirac-dirac}, and where the equality holds in $\mathcal M^1_n(D)$. We have for all $x\in D$:
\[
\sum_{\eta\in\mathcal M^1_n(D)} \sum_{y\in D}|w(\delta_y)\mathbbm 1_y(\eta) \psi_x^r(\eta)| = \sum_{y\in D} |w(\delta_y)|\psi_x^r(\delta_y) \leqslant \|w\|_1 <\infty,
\]
and hence the following computation holds:
\begin{align*}
0 &= \langle w,\psi_x^r \rangle = \sum_{\eta\in\mathcal M^1_n(D)}w(\eta) \psi_x^r(\eta) \\&= \sum_{\eta\in\mathcal M^1_n(D)}\po \sum_{y\in D}w(\delta_y)\mathbbm 1_y(\eta) \pf \psi_x^r(\eta) = \sum_{y\in D} w(\delta_y)\psi_x^r(\delta_y) = w(\delta_x).
\end{align*}
Thus, $w=0$, which concludes for the uniqueness.
\end{proof}

In the proof of Theorem~\ref{thm:soft-kill-1}, in order to apply Lemma~\ref{lem:poisson-eq}, we need to show that for all $x\in D$, 
\[
\langle \partial_t u_0 - L_m^*u_0 ,\psi_x^r \rangle = 0.
\]
To this end, let us introduce
\[
\delta_{x;y} = \frac{n-1}{n}\delta_x + \frac{1}{n}\delta_y \in \mathcal M^1_n(D).
\]
The next lemma provides an explicit formula for $\psi_x^r(\delta_{x;y})$.
The fact that we only need the values of the committor functions on the set $\Delta \cup\left\{ \delta_{x;y}; x,y\in D \right\}$ is a consequence of the fact that, as explained above, in the limiting regime, the selection dynamics is much faster than the mutation dynamics, so that the support of $\pi_t$ is concentrated on $\{x,y\}$
in the transition from $\delta_x$ to $\delta_y$. After one particle jumped from $x$ to $y$, the probability that all the particles go back to $x$ is $\psi^r_x(\delta_{x;y})$, and the probability that all the particles jumps to $y$ is $\psi^r_y(\delta_{x;y})$.

\begin{lem}\label{lem:game}
For all $x,y\in D$, $x\neq y$, the committor functions $\psi_x^r$ satisfy:
\[
\psi_x^r(\delta_{x;y}) = \alpha_{x,y,r}\frac{\alpha_{x,y,r}^{n-1}-1}{\alpha_{x,y,r}^n-1},\qquad \psi_x^r(\delta_{y;x}) =  \alpha_{x,y,r}^{n-1}\frac{\alpha_{x,y,r}-1}{\alpha_{x,y,r}^n-1},
\]
where $\alpha_{x,y,r}$ is defined in~\eqref{eq:coef_lim}.
\end{lem}

\begin{proof}
The selection dynamics cannot extend the support of a measure, hence if ${\rm supp}\, \pi^0_0 = \left\{x,y\right\}$, then ${\rm supp}\, \pi^0_t \subset \left\{x,y\right\}$ for all $t\geqslant 0$ and $Z_t = \pi^0_t(x)$ completely determines $\pi^0_t$. The process $Z$ corresponds to a Gambler's ruin problem, namely a Markov process on $\left\{0,...,n\right\}$ with jump rates \[
p(k,k+1) = \alpha_r(y)\frac{k(n-k)}{n-1},\quad p(k,k-1) = \alpha_r(x)\frac{k(n-k)}{n-1},
\]
where $\alpha_r$ is given by~\eqref{eq:normalised-coef}. Denote by $R$ its generator, 
\[
\tau_i(Z) = \inf\left\{t\geqslant 0, Z_t = i \right\},\quad g(k) = \mathbb P_k(\tau_n(Z) < \tau_0(Z)).
\]
Then, $g$ is solution to the problem:
\[
Rg = 0, \, g(0) = 0, \, g(n) = 1.
\]
We have:
\[
Rg(k) = \alpha_r(y)\frac{k(n-k)}{n-1} \po g(k+1) - g(k) \pf + \alpha_r(x)\frac{k(n-k)}{n-1} \po g(k-1) - g(k) \pf.
\]
Hence:
\[
g(k+1) - g(k) = \alpha_{x,y,r}^{-1}\po g(k) - g(k-1) \pf,
\]
$\Leftrightarrow$
\[
g(k) - g(k-1) = \alpha_{x,y,r}^{-(k-1)}\po g(1) - g(0) \pf = \alpha_{x,y,r}^{-(k-1)}g(1),
\]
$\Leftrightarrow$
\[
g(k) = \frac{\alpha_{x,y,r}^{-k}-1}{\alpha_{x,y,r}^{-1}-1} g(1).
\]
The condition $g(n)=1$ yields:
\[
g(1) = \frac{\alpha_{x,y,r}^{-1} - 1}{\alpha_{x,y,r}^{-n}-1},
\]
and finally:
\[
\psi_x^r(\delta_{x;y}) = g(n-1) = \frac{\alpha_{x,y,r}^{-(n-1)}-1}{\alpha_{x,y,r}^{-n}-1},\quad  \psi_x^r(\delta_{y;x}) = g(1) = \frac{\alpha_{x,y,r}^{-1}-1}{\alpha_{x,y,r}^{-n}-1},
\]
which concludes the proof.
\end{proof}

\subsection{Proof of Theorem~\ref{thm:soft-kill-1}}\label{sec:ProofT1}

The proof of Theorem~\ref{thm:soft-kill-1} requires a last technical result on the convergence of jump Markov processes when their jump rates converge. We recall that the rates $\tilde q_r$ and $\tilde q_{\infty}$ are respectively defined by~\eqref{eq:jump-rates-r} and~\eqref{eq:jump-rates-lim}.

\begin{lem}\label{lem:conv-Yr-to-Y}
Let $Y^r$ and $Y$ be jump Markov processes on $D$ with rates $\tilde q_r$ and $\tilde q_{\infty}$ respectively. Under Assumption~\ref{assu:assu1}
\[
\lim_{r\rightarrow\infty}\|\mathcal Law (Y_0^r) - \mathcal Law (Y_0)\|_{TV} = 0 \implies \forall t\geqslant 0,\, \lim_{r\rightarrow\infty}\|\mathcal Law (Y_t^r) - \mathcal Law (Y_t)\|_{TV} = 0.
\]
\end{lem}

\begin{proof}
We have that for all $x,y\in D$, 
\[
\lim_{r\rightarrow\infty}\tilde q_r(x,y) = \tilde q_{\infty}(x,y),
\]
and for all $x\in D$,
\[
\sum_{y\in D}\sup_{r\in \N}\tilde q_r(x,y) < \infty.
\]
Hence, \cite[Theorem~19.25, p.~385]{Kal02} yields the result.
\end{proof}

We are now in position to prove Theorem~\ref{thm:soft-kill-1}.
\begin{proof}[Proof of Theorem~\ref{thm:soft-kill-1}]
Recall the jump rates:
\[
\tilde q_{r}(x,y) = \left\{ \begin{aligned} n &q(x,y) \frac{\alpha_{x,y,r}-1}{\alpha_{x,y,r}^n-1} &\text{ if }\alpha_{x,y,r} \neq 1, \\ &q(x,y)&\text{ otherwise,} 
\end{aligned}\right.
\]
and let $Y^{r}$ be the Markov process with jump rates $(\tilde q_{r}(x,y))_{x,y\in D}$ and initial condition $Y_0^r$ such that the law of $\delta_{Y_0^r}$ is $\mathcal Law(\pi_{\tau}^0)$. Write $p_x^{r}(t) = \mathbb P\po Y_t^r = x \pf$, for $x\in D$. The $(p_x^{r})_{x\in D}$ satisfy:
\[
\sum_{x\in D} p_x^{r} = 1,
\]
and the Kolmogorov equations:
\[
\partial_t p_x^{r} = \sum_{y\neq x} \po \tilde q_{r}(y,x) p_y^{r} - \tilde q_{r}(x,y) p_x^{r} \pf.
\]
Denote by $u(t) = \po \mathbb P(\pi_t = \xi )\pf_{\xi\in \mathcal M^1_n(D)}$ the law of the empirical measure of the Fleming--Viot process (which depends on $r$, even though this is not indicated explicitly for the sake of notation) and define the expected limit of $u$ (see Equation~\eqref{eq:u}) by:
\[
u_0(t) = \sum_{x\in D} p_x^{r}(t) \mathbbm 1_x,
\]
which corresponds to the law of $\delta_{Y^r_t}$. Lemma~\ref{lem:topology} yields that that for all $t\geqslant 0$:
\[
L^*_{s,0} u_0(t) = 0.
\]

The goal now is to define $u_1$ as a solution to~\eqref{eq:Poisson} using Lemma~\ref{lem:poisson-eq}. In order to do so, let us show that \[\partial_t u_0 - L_m^*u_0 \in \po \mathrm{Ker} L_{s,0} \pf^{\perp}.\] Fix $x\in D$. First:
\[
\langle \partial_t u_0 ,\psi_x^r \rangle = \partial_t p_x^r.
\]
For all $f\in \ell^1(\mathcal M^1_n(D))$:
\[
L_m^*f(\xi) = \sum_{x\neq y} \po\po n\xi(y) + 1 \pf q(y,x) f\left(\xi +\frac{\delta_y - \delta_x}{n}\right) - n\xi(x)q(x,y)f\po \xi \pf\pf,
\]
hence for all $x\neq y$, $n\geqslant 3$:
\[
L_m^*u_0(\eta) = \begin{cases}
  -n\sum_{z\neq x} q(x,z)p^r_x & \text{if $\eta=\delta_x$,}\\
  nq(x,y)p^r_x & \text{if $\eta=\delta_{x;y}$,}\\
  0 & \text{otherwise.}
\end{cases}
\]
In the case $n=2$, the formula is adapted by $L_m^*u_0(\delta_{x;y})= nq(x,y)p^r_x + nq(y,x)p^r_y$.
Using Lemma~\ref{lem:game}, this yields that:
\begin{align*}
&\langle L^*_mu_0,\psi_x^r \rangle \\&= L_m^*u_0(\delta_x) \psi_x^r(\delta_x) + \sum_{y\neq x } \po L_m^*u_0(\delta_y) \psi_x^r(\delta_y) + L_m^*u_0(\delta_{x;y}) \psi_x^r(\delta_{x;y}) + L_m^*u_0(\delta_{y;x}) \psi_x^r(\delta_{y;x}) \pf \\ &= -n\sum_{z\neq x} q(x,z)p_x^{r} + \sum_{y\neq x} nq(x,y)p_x^{r} \alpha_{x,y,r}\frac{\alpha_{x,y,r}^{n-1}-1}{\alpha_{x,y,r}^n-1} + nq(y,x)p_y^{r} \alpha_{x,y,r}^{n-1}\frac{\alpha_{x,y,r}-1}{\alpha_{x,y,r}^n-1} \\ &= \sum_{y\neq x} \po -\tilde q_{r}(x,y)p_x^{r} + \tilde q_{r}(y,x)p_y^{r} \pf.
\end{align*}
Finally we get that for all $x\in D$:
\[
\langle \partial_t u_0 - L_m^*u_0,\psi_x^r \rangle = \partial_t p_x^{r} - \sum_{y\neq x} \po \tilde q_{r}(y,x)p_y^{r} - \tilde q_{r}(x,y)p_x^{r} \pf =  0.
\]
Lemma~\ref{lem:topology} then yields that $\partial_t u_0 - L_m^*u_0 \in \po \mathrm{Ker} L_{s,0} \pf^{\perp}$.
Thus, Lemma~\ref{lem:poisson-eq} yields that, for any $t \geqslant 0$, there exists $u_1(t) \in (\mathrm{Ker} L_{s,0})^\perp$ such that
\[
    \partial_t u_0(t) - L_m^*u_0(t) = L_{s,0}^* u_1(t),
\]
with the explicit representation
\[
    u_1(t) = \int_0^\infty \po \partial_t u_0(t) - L_m^*u_0(t)\pf e^{sL_{s,0}} \dd s.
\]
From this representation and the regularity of $u_0(t)$ provided by its definition, one gets that $u_1$ is a $\mathcal{C}^1$ function of $t$, and that $\partial_t u_1$ solves the Poisson equation
\[
    \partial^2_t u_0(t) - L_m^* \partial_t u_0(t) = L_{s,0}^* \partial_t u_1(t).
\]
As a conclusion, using Lemma~\ref{lem:poisson-eq} again, we deduce that
\[
\|u_1(t)\|_{1} \leqslant C \|\partial_t u_0(t) - L_m^*u_0(t)\|_{1},\quad \|\partial_t u_1(t)\|_{1} \leqslant C \|\partial_t^2 u_0(t) - L_m^*\partial_t u_0(t)\|_{1}.
\]
The Kolmogorov equations yield that $\|\partial_t u_0\|_{1}\leqslant Q$, and $\|\partial_t^2 u_0\|_{1}\leqslant Q^2$, and since $L_m^*$ is a bounded operator, we have that:
\begin{equation}\label{eq:bounds_u1}
\sup_{r\in\N}\sup_{t\geqslant 0} \|u_1(t)\|_{1} <\infty,\quad \sup_{r\in\N}\sup_{t\geqslant 0} \|\partial_t u_1(t)\|_{1} <\infty.
\end{equation}

Let us define $v$ by: 
\begin{equation}\label{eq:decompos_u}
u = u_0 + \frac{1}{\underline\lambda_r}u_1 + v.
\end{equation}
Then $v$ satisfies the following equation:
\[
\partial_t v = L^*v + \frac{1}{\underline\lambda_r} \po L^*_m u_1 - \partial_t u_1 \pf, \, v(0) = u(0) - u_0(0) -\frac{1}{\underline\lambda_r}u_1(0). 
\]
In the case where there exists $z\in D$ such that $\pi(\X_0)=\delta_z$, then $\tau=0$ so $Y^r_0=z$, and we have $v(0) = -u_1(0)/\underline\lambda_r$. Thus we can write:
\begin{equation}\label{eq:reste}
v(t) =-\frac{1}{\underline\lambda_r}e^{tL^*}u_1(0) + \frac{1}{\underline\lambda_r}\int_0^t e^{(t-s)L^*} \po L^*_m u_1 - \partial_t u_1 \pf \dd s,
\end{equation}
where for any $w\in\ell^1(\mathcal M_n^1(D))$, $\eta\in\mathcal M_n^1(D)$
\[
e^{tL^*}w(\eta)= \sum_{\xi\in\mathcal M^1_n(D)}w(\xi)\mathbb P_{\xi}\po \pi_t = \eta \pf.
\]
Thanks to the bounds on $\|u_1\|_1$ and $\|\partial_t u_1\|_1$, Equation~\eqref{eq:reste} yields that for any $T>0$, there exists $C>0$ such that:
\begin{equation}\label{eq:bounds_v}
\sup_{t \in [0,T]} \|v(t)\|_{1} \leqslant C/\underline\lambda_r.
\end{equation}
As a conclusion, from~\eqref{eq:decompos_u}, for any $t \in [0,T]$,
\[
\| \mathcal Law\po \pi(\X_t) \pf - \mathcal Law\po \delta_{Y_t^r} \pf\|_1 \leqslant \|v(t)\|_{1} + \frac{\|u_1(t)\|_{1}}{\underline\lambda_r},
\]
and the bounds~\eqref{eq:bounds_u1} and~\eqref{eq:bounds_v} thus yield the second statement of Theorem~\ref{thm:soft-kill-1}.

To prove the first statement, with an initial condition which is not necessarily a Dirac mass, we first write
\[
\| \mathcal Law\po \pi(\X_t) \pf - \mathcal Law\po \delta_{Y_t} \pf\|_1 \leqslant \| \mathcal Law\po \pi(\X_t) \pf - \mathcal Law\po \delta_{Y_t^r} \pf\|_1 + \| \mathcal Law\po \delta_{Y_t^r}\pf - \mathcal Law\po \delta_{Y_t} \pf\|_1,
\]
where $Y$ is the continuous-time Markov chain with initial condition $\eta_\infty$ given by Lemma~\ref{lem:initial-condi} and jump rates $\tilde{q}_\infty(x,y)$. For the first term of the right-hand side, we integrate in time the equation satisfied $v$ between  $1/\sqrt{\underline\lambda_r}$ and $t$ to get:
\[
v(t) = e^{(t-1/\sqrt{\underline\lambda_r})L^*}v\po \frac{1}{\sqrt{\underline\lambda_r}} \pf + \frac{1}{\underline\lambda_r}\int_{\frac{1}{\sqrt{\underline\lambda_r}}}^t e^{(t-s)L^*} \po L^*_m u_1 - \partial_t u_1 \pf \dd s.
\]
The second term of this last equality still goes to $0$ at speed $1/\underline\lambda_r$, and we are just left with showing that:
\[
v\po \frac{1}{\sqrt{\underline\lambda_r}} \pf = u\po \frac{1}{\sqrt{\underline\lambda_r}} \pf - u_0\po \frac{1}{\sqrt{\underline\lambda_r}} \pf -\frac{1}{\underline\lambda_r}u_1\po \frac{1}{\sqrt{\underline\lambda_r}} \pf \underset{r\rightarrow\infty}{\rightarrow} 0,
\]
in the $\ell^1$ distance.
Since $u_1$ is bounded in $\ell^1(\mathcal M_n^1(D))$, uniformly in $t\geqslant 0$ and $r\in \N$, we have that: 
\[
\lim_{r\rightarrow\infty}\left\|\frac{1}{\underline\lambda_r}u_1\po \frac{1}{\sqrt{\underline\lambda_r}}\pf\right\|_{1} = 0.
\]
Moreover, $Y^{r}$ is a continuous-time Markov chain with initial condition $Y_0^r$, where $Y_0^r$ is such that $\mathcal Law(\delta_{Y_0^r}) = \mathcal Law(\pi^0_\tau)$, and jump rates uniformly bounded with respect to $r$, hence $\underline\lambda_r\rightarrow\infty$ yields that:
\[
 \lim_{r\rightarrow\infty} \left\| \mathcal Law(\pi^0_\tau) - u_0\po \frac{1}{\sqrt{\underline\lambda_r}} \pf \right\|_{1} = 0.
\]
Let us now show that 
\[
\lim_{r\rightarrow\infty}\|u\po 1/\sqrt{\underline\lambda_r} \pf - \mathcal Law(\pi^0_\tau) \|_1 =0,
\]
even though $u(0) = \eta$ for all $r\in\N$. Write 
\begin{equation}
\tau_1 = \inf\left\{ t\geqslant 0, \pi_t \in \Delta \right\},
\end{equation}
and $S_1$ for the first jump of a particle due to the mutation mechanism. 
Using the notation of Lemma~\ref{lem:different-time-scale}, write $\Theta = {\rm supp}(\eta)$, $b_r = 1$, $L_r^2=L_m$, $L_r^1 = L_{s,0}^{{\rm supp}(\eta)} $, where 
\[
L_{s,0}^{{\rm supp}(\eta)}f(\xi) = \sum_{x,y\in {\rm supp}(\eta) } \alpha_r(x) \xi(x)\xi(y) \left( f\left(\xi +\frac{\delta_y - \delta_x}{n}\right) - f(\xi) \right)
\]
as well as $a_r =\underline\lambda_r$, and $c_r=\sqrt{\underline\lambda_r}$.
In this case, we have $\mathbb M^1 = \Delta$, and Lemma~\ref{lem:different-time-scale} yields that:
\begin{equation}\label{eq:conv-condi-init}
\lim_{r \rightarrow \infty} \mathbb P_{\eta}\po \tau_1 < \frac{1}{\sqrt{\underline\lambda_r}} < S_1 \pf = 1.
\end{equation}
For all $\xi\in\mathcal M^1_n(D)$, we may write:
\begin{align*}
    u\po \frac{1}{\sqrt{\underline\lambda_r}} \pf(\xi) &= \mathbb P_{\eta}\po \pi_{\frac{1}{\sqrt{\underline\lambda_r}}} =\xi, \tau_1 < \frac{1}{\sqrt{\underline\lambda_r}} < S_1 \pf +  \mathbb P_{\eta}\po \pi_{\frac{1}{\sqrt{\underline\lambda_r}}} =\xi, \po\tau_1 < \frac{1}{\sqrt{\underline\lambda_r}} < S_1 \pf^c\pf  \\ &= \mathbb P_{\eta}\po \pi^0_{\tau} =\xi, \tau_1 < \frac{1}{\sqrt{\underline\lambda_r}} < S_1 \pf +  \mathbb P_{\eta}\po \pi_{\frac{1}{\sqrt{\underline\lambda_r}}} =\xi, \po\tau_1 < \frac{1}{\sqrt{\underline\lambda_r}} < S_1 \pf^c\pf,
\end{align*}
and thus
\begin{align*}
    \|u\po 1/\sqrt{\underline\lambda_r} \pf - \mathcal Law(\pi^0_\tau) \|_1 \leqslant 2\mathbb P_{\eta} \po \po\tau_1 < \frac{1}{\sqrt{\underline\lambda_r}} < S_1 \pf^c \pf \underset{r\rightarrow\infty}{\rightarrow}0.
\end{align*}
Finally, Lemma~\ref{lem:initial-condi} and Lemma~\ref{lem:conv-Yr-to-Y} imply the convergence of $\| \mathcal Law\po \delta_{Y_t^r}\pf - \mathcal Law\po \delta_{Y_t} \pf\|_1$ towards $0$, and conclude the proof of Theorem~\ref{thm:soft-kill-1}.
\end{proof}

\section{Proof of Theorem~\ref{thm:conv-trajectories}}\label{sec:conv-trajectories}

We are now interested in the convergence of the trajectories, in the space $\mathcal D([0,T],\mathcal M^1(D))$ of $\mathcal M^1(D)$-valued cadlag paths. As stated in the introduction, there is no hope of having convergence in the Skorohod topology. Instead, we use the $L^1_{TV}$ topology. 
Let $F:\mathcal D([0,T],\mathcal M^1(D))\to\R$ be bounded and Lipschitz continuous for the $L^1_{TV}$ distance~\eqref{eq:distance-L1}. We decompose
\begin{align*}
    &\E \po F\po \po\pi(\X^r_t)\pf_{t\in [0,T]} \pf \pf - \E\po F \po \po \delta_{Y_t}\pf_{t\in [0,T]} \pf \pf \\
    &= \left(\E \po F\po \po\pi(\X^r_t)\pf_{t\in [0,T]} \pf \pf - \E\po F \po \po \delta_{Y_t^r}\pf_{t\in [0,T]} \pf \pf\right) \\
    &\quad + \left(\E\po F \po \po \delta_{Y_t^r}\pf_{t\in [0,T]} \pf \pf - \E\po F \po \po \delta_{Y_t}\pf_{t\in [0,T]} \pf \pf\right),
\end{align*}
with $\mathcal{L}aw(Y^r_0)=\mathcal{L}aw(\pi^0_\tau)$ and $\mathcal{L}aw(Y_0)=\eta_\infty$. Thus, the conclusion of Theorem~\ref{thm:conv-trajectories} follows from the combination of Lemmas~\ref{lem:conv-Yr-to-Y-traj} and~\ref{lem:coupling-traj}.

\begin{lem}\label{lem:conv-Yr-to-Y-traj}
    The process $(\delta_{Y_t^r})_{t\in [0,T]}$ converges in distribution, in $\mathcal D([0,T],\mathcal M^1(D))$, to the process $(\delta_{Y_t})_{t\in [0,T]}$.
\end{lem}
\begin{proof}
    Let us endow the set $D$ with the distance $\rho(x,y)=\mathbbm{1}_{x\not=y}$, and the set $\mathcal{D}([0,T],D)$ of cadlag trajectories $[0,T] \to D$ with the associated Skorohod $J_1$-topology.  By~\cite[Theorem~19.25, p.~385]{Kal02} (see the proof of Lemma~\ref{lem:conv-Yr-to-Y}), the process $(Y^r_t)_{t \in [0,T]}$ converges in distribution to $(Y_t)_{t \in [0,T]}$ in this space. To conclude, we now check that the mapping
    \[
        \begin{array}{rcl}
            \mathcal{D}([0,T],D) & \to & \mathcal D([0,T],\mathcal M^1(D))\\
            (y_t)_{t \in [0,T]} & \mapsto & (\delta_{y_t})_{t \in [0,T]}
        \end{array}
    \]
    is continuous. Let $((y^k_t)_{t \in [0,T]})_{k \geqslant 1}$ be a sequence of elements of $\mathcal{D}([0,T],D)$ which converges to $(y_t)_{t \in [0,T]}$ in the $J_1$-topology. We have
    \[
        \|(\delta_{y^k_t})_{t \in [0,T]} - (\delta_{y_t})_{t \in [0,T]}\|_{L^1_{TV}} = 2\int_0^T \mathbbm{1}_{y^k_t \not= y_t}\dd t = 2 \int_0^T \rho(y^k_t, y_t)\dd t.
    \]
    By~\cite[p.~125]{Billingsley}, the convergence of $(y^k_t)_{t \in [0,T]}$ in the $J_1$-topology implies the convergence of $\rho(y^k_t,y_t)$ to $0$ for almost every $t \in [0,T]$, and therefore the conclusion follows from the dominated convergence theorem.
\end{proof}

\begin{lem}\label{lem:coupling-traj}
    There exists a coupling of the processes $\po\pi(\X^r_t)\pf_{t\in [0,T]}$ and $\po \delta_{Y_t^r}\pf_{t\in [0,T]}$ such that
\[
\lim_{r\rightarrow\infty} \mathbb{E}\left[\int_0^T \|\pi(\X^r_t)-\delta_{Y_t^r}\|_{TV} \, \dd t\right] = 0.
\] 
\end{lem}
\begin{proof}
Fix $T>0$. The proof of Lemma~\ref{lem:coupling-traj} is divided into five steps. The first one is dedicated the definition of the coupling.

\noindent \emph{Step~1.} Let us recall the notation~\eqref{eq:empiri-meas} and introduce $\sigma_0 = 0$,
\[
\tau_k = \inf\left\{t\geqslant \sigma_k, \pi_t \in \Delta\right\},
\]
and
\[
\sigma_{k+1} = \inf\left\{t\geqslant \tau_k, \pi_t \notin \Delta\right\},
\]
for all $k\in \N$. We define an intermediate process, which is not Markov: 
\[
\bar \eta_t = \pi_{\tau_{N(t)}}, \quad N(t) = \max\left\{ k, \tau_k \leqslant t \right\},
\]
and for $t\in [0,\tau_0]$, write $\bar \eta_t = \pi_{\tau_0}$. This is a process living in $\Delta$, and we will get a Markov process by a time re-scaling and conditioning. Define the change of time:
\[
s(t) = \inf\left\{s\geqslant 0, \int_0^s \mathbbm 1_{\pi_u\in \Delta} \dd u > t \right\}.
\]
Write $\tilde N(t) = \max\left\{ k, \sigma_k \leqslant t \right\}$, and
\begin{multline*}
\mathcal A_r(T) \\= \left\{ \text{For all }k\in \llbracket 0,\tilde N(T)\rrbracket, \text{ there are no jumps from the mutation mechanism between }\sigma_k\text{ and }\tau_k  \right\}.
\end{multline*}
If $\pi_{\tau_k}= \delta_x$, then $\sigma_{k+1}-\tau_k$ is an exponential random variable with rate $n\sum_{y\neq x}q(x,y)$. Additionally, on $\mathcal A_r(T)$, Lemma~\ref{lem:game} yields that the Dirac mass goes from $x$ to $y$ with probability
\[
\frac{q(x,y)}{\sum_{y\neq x}q(x,y)}\alpha_{y,x,r}^{n-1}\frac{\alpha_ {y,x,r}-1}{\alpha_{y,x,r}^n-1}.
\]
Hence, writing
\[
\delta_{Y^{r,1}_t } = \bar \eta_{s(t)},
\]
we have that on $\mathcal A_r(T)$, $Y^{r,1}$ is exactly a continuous-time Markov chain with jump rates $\tilde q_{r}$ and initial condition $\pi_{\tau_0}$. Let $Y^{r,2}$ be another continuous-time Markov chain on $D$ with jump rates $\tilde q_{r}$, initial condition with distribution $\mathcal{L}aw(\pi^0_\tau)$, independent of the Fleming--Viot process, and write:
\[
Y^{r} = Y^{r,1}\mathbbm 1_{\mathcal A_r(T)} + Y^{r,2}\mathbbm 1_{\mathcal A^c_r(T)}.
\]
The process $Y^{r}$ is a jump Markov process with jump rates $\tilde q_{r}$ and initial condition with distribution $\mathcal{L}aw(\pi^0_\tau)$. In Step~2, we show that the probability of $\mathcal{A}_r(T)$ converges to $1$. In Step~3, we show that $\pi$ is close to $\bar \eta$, and finally in Steps~4 and ~5, we show that $\bar\eta$ is close to $Y^{r,1}$ in the $L^1_{TV}$ distance defined in~\eqref{eq:distance-L1}, and thus to $Y^r$.

\medskip
\noindent \emph{Step~2.} Let us show that for all $T>0$, 
\[
\lim_{r\rightarrow\infty} \mathbb P\po \mathcal A_r(T) \pf = 1.
\]
For any $k \geqslant 0$, the two events $\left\{ \tilde N(T)\geqslant k \right\}$ and $\left\{ \text{there is a mutation jump between }\sigma_k \text{ and }\tau_k\right\}$ are independent. Thus we have:
\begin{align*}
    \mathbb P&\po \mathcal A_r(T)^c \pf \\&\leqslant  \sum_{k=0}^{\infty} \mathbb P\po \tilde N(T) \geqslant k, \text{ there is a mutation jump between }\sigma_k \text{ and }\tau_k \pf \\ & \leqslant \sup_{k\in \N}\mathbb P\left( \text{there is a mutation jump between }\sigma_k \text{ and }\tau_k \right) \sum_{k=0}^{\infty} \mathbb P\po \tilde N(T) \geqslant k \pf.
\end{align*}
Because $\sigma_{k+1}-\tau_k$ is an exponential random variable with parameter bounded by $nQ$ (recall the definition~\eqref{eq:assQlambda} for $Q$), $\tilde N(T)$ is bounded by a Poisson random variable with parameter $nQT$, and hence:
\[
\sum_{k=0}^{\infty} \mathbb P\po \tilde N(T) \geqslant k \pf = \E\po \tilde N(T) \pf \leqslant nQT.
\]
The convergence~\eqref{eq:time-scale} and the strong Markov property at time $\sigma_k$ yield that
\[
\lim_{r\rightarrow\infty} \sup_{k\in \N}\mathbb P\left( \text{there is a mutation jump between }\sigma_k \text{ and }\tau_k \right) = 0,
\]
which concludes the proof.

\medskip
\noindent \emph{Step~3.} Let us now prove that the expected $L^1_{TV}$ distance between $\pi$ and $\bar \eta$ goes to $0$ when $r$ goes to infinity.
We have:
\[
\|\pi-\bar \eta\|_{L^1_{TV}}\mathbbm 1_{\mathcal{A}_r(T)} \leqslant 2 \mathbbm 1_{\mathcal{A}_r(T)}\sum_{k=0}^{\tilde N(T)} \bar \tau_k - \sigma_k \leqslant 2\sum_{k=0}^{\infty} \po \bar \tau_k - \sigma_k\pf \mathbbm 1_{\sigma_k \leqslant T}, 
\]
where 
\begin{equation}\label{eq:new-death}
\bar \tau_k = \inf\left\{t\geqslant \sigma_k, \bar \pi^k_t \in \Delta\right\},
\end{equation}
and $\bar \pi^k$ is equal to $\pi$ on $[0,\sigma_k]$, and on $[\sigma_k,\infty]$, $\bar \pi^k$ is a Markov process with generator $L_{s}$, constructed with the same exponential variables as the ones used for the selection mechanism in the dynamics of $\pi$. The application of Lemma~\ref{lem:different-time-scale} to $(\bar\pi^k_{t+\sigma_k})_{t\geqslant 0}$ yields that there exist $C,c>0$, independent of $k\in\N$, such that for all $t\geqslant 0$:
\[
\mathbb P\po \bar \tau_k - \sigma_k >t \middle| \sigma_k \leqslant T  \pf \leqslant Ce^{-c\underline \lambda_r t},
\]
yielding that
\[
\lim_{r\rightarrow\infty} \sup_{k\in \N}\E\po  \bar \tau_k - \sigma_k  \middle| \sigma_k \leqslant T \pf = 0.
\]
Using again that $\tilde N(T)$ is bounded by a Poisson random variable with parameter $nQT$ we get:
\begin{align*}
\E\po \|\pi-\bar \eta\|_{L^1_{TV}} \pf &\leqslant 2T\mathbb P\po \mathcal A_r(T)^c \pf +  2\sup_{k\in \N}\E\po \bar \tau_k - \sigma_k \middle| \sigma_k \leqslant T \pf \sum_{k\geqslant 0} \mathbb P\po \tilde N(T) \geqslant k \pf \\ &= 2T\mathbb P\po \mathcal A_r(T)^c \pf +  2\sup_{k\in \N}\E\po \bar \tau_k - \sigma_k \middle| \sigma_k \leqslant T \pf \E(\tilde N(T)) \\ &\leqslant 2T\mathbb P\po \mathcal A_r(T)^c \pf +  2n Q T \sup_{k\in \N}\E\po \bar\tau_k - \sigma_k \middle|\sigma_k \leqslant T \pf \underset{r \rightarrow \infty}{\rightarrow} 0,
\end{align*}
where we used the result of Step~2 to control the term $\mathbb P\po \mathcal A_r(T)^c \pf$.

\medskip
\noindent \emph{Step~4.} In order to show that the distance between $\bar \eta$ and  $(\delta_{Y^{r,1}_t})_{0\leqslant t \leqslant T}$ goes to zero as well, let us define the event
\[
\mathcal B_r(T) = \left\{ \sum_{k=0}^{\tilde N(T)}\tau_k - \sigma_k < \min_{0\leqslant k \leqslant \tilde N(T)} \sigma_{k+1}-\tau_k \right\},
\]
and show that
\[
\lim_{r\rightarrow\infty} \mathbb P\po \mathcal A_r(T), \mathcal B_r(T) \pf = 1.
\]
Fix $\varepsilon>0$, and $N_0\in \N$ such that $\mathbb P\po \tilde N(T) >N_0\pf <\varepsilon$. Then
\[
\mathbb P\po \mathcal B_r(T)^c,\mathcal A_r(T) \pf \leqslant \varepsilon + \mathbb{P}\po \sum_{k=0}^{N_0}\tau_k - \sigma_k > \min_{0\leqslant k \leqslant  N_0} \sigma_{k+1}-\tau_k ,\mathcal A_r(T) \pf.
\]
The random variable $\min_{0\leqslant k \leqslant  N_0} \sigma_{k+1}-\tau_k$ is lower bounded by a exponential random variable $E$ with parameter $(N_0+1)nQ$, independent of the family of random variable $(\tau_k-\sigma_k)_{0\leqslant k \leqslant N_0}$. This implies that
\begin{align*}
    \mathbb{P}\po \sum_{k=0}^{N_0}\tau_k - \sigma_k > \min_{0\leqslant k \leqslant  N_0} \sigma_{k+1}-\tau_k ,\mathcal A_r(T) \pf &\leqslant \mathbb{P}\po \exists k\in \cco 0,N_0\ccf,\, \bar\tau_k - \sigma_k > E/(N_0+1) \pf \\&\leqslant (N_0+1) \sup_{k\in\N}\mathbb{P}\po\bar \tau_k - \sigma_k > E/(N_0+1) \pf.
\end{align*}
By conditioning on $E$, Equation~\eqref{eq:bound-death-time} then yields
\[
\sup_{k\in\N}\mathbb{P}\po\bar \tau_k - \sigma_k > E/(N_0+1) \pf \leqslant C\E\po e^{-c\underline \lambda_r E/(N_0+1)} \pf = C \frac{nQ(N_0+1)^2}{nQ(N_0+1)^2 + c\underline \lambda_r} \underset{r\rightarrow\infty}{\rightarrow}0,
\]
so that
\[
\limsup_{r\rightarrow\infty} \mathbb P\po \mathcal A_r(T), \mathcal B_r(T)^c \pf \leqslant \varepsilon,
\]
for all $\varepsilon>0$, which concludes Step~4.

\medskip
\noindent \emph{Step~5.} 
By the result of Step~3, all is left to do is to show that 
\[
\lim_{r\rightarrow\infty} \|\bar \eta -(\delta_{Y^{r}_t})_{0\leqslant t \leqslant T}\|_{L^1_{TV}} = 0.
\]
To proceed we shall use the result of Step~4. First notice that for all $s\geqslant 0$
\[
\int_0^s \mathbbm 1_{\pi_u\in \Delta} \dd u  = s - \sum_{k=0}^{\tilde N(s)} s\wedge \tau_k - \sigma_k,
\]
so that
\[
s(t) = t + \sum_{k=0}^{\tilde N(s(t))} \tau_k - \sigma_k \geqslant t.
\]
Additionally, $\delta_{Y^{r,1}_t} \neq \bar \eta_t$ only if there is $p\in \N$ such that $t \leqslant \sigma_p <s(t)$, and on the event $\mathcal B_r(T)$, $t \leqslant \sigma_p < s(t)$ implies that $\tilde N(s(t)) = p$. Hence 
\[
\sigma_p - \sum_{k=0}^p \tau_k - \sigma_k\leqslant t \leqslant \sigma_p . 
\]
Hence the total amount of time where $\delta_{Y^{r,1}} \neq \bar \eta$ is bounded as follow
\[
\mathbbm 1_{\mathcal A_r(T),\mathcal B_r(T)} \int_0^T \mathbbm 1_{\delta_{Y^{r,1}_t} \neq \bar \eta_t} \dd t \leqslant \mathbbm 1_{\mathcal A_r(T),\mathcal B_r(T)}\sum_{p=0}^{\tilde N(T)}\sum_{k=0}^p \tau_k - \sigma_k \leqslant \sum_{k=0}^{\tilde N(T)} (\tilde N(T)-k)\po \bar \tau_k - \sigma_k\pf,
\]
where $\bar \tau_k$ was defined in~\eqref{eq:new-death}. We have:
\[
\|\delta_{Y^{r}} - \bar \eta \|_{L^1_{TV}}\mathbbm 1_{\mathcal A_r(T)}\mathbbm 1_{\mathcal B_r(T)} = \|\delta_{Y^{r,1}} - \bar \eta \|_{L^1_{TV}}\mathbbm 1_{\mathcal A_r(T)}\mathbbm 1_{\mathcal B_r(T)} \leqslant 2\sum_{k=0}^{\infty} \tilde N(T) \po \bar \tau_k - \sigma_k\pf \mathbbm 1_{\sigma_k \leqslant T}.
\]
Using Cauchy-Schwarz inequality, one gets:
\begin{multline*}
\E\po\sum_{k=0}^{\infty} \tilde N(T) \po \bar \tau_k - \sigma_k\pf \mathbbm 1_{\sigma_k \leqslant T} \pf \leqslant \sum_{k=0}^{\infty} \E\po \tilde N(T)^2\pf^{\frac{1}{2}} \E\po \po \bar \tau_k - \sigma_k\pf^2 \mathbbm 1_{\sigma_k \leqslant T}\pf^{\frac{1}{2}} \\\leqslant \E\po\tilde N(T)^2\pf^{\frac{1}{2}} \po \sup_{k\in \N}\E\po \po\bar\tau_k - \sigma_k\pf^2 \middle|\sigma_k \leqslant T \pf\pf^{\frac{1}{2}} \sum_{k\in\N}\sqrt{\mathbb P\po \tilde N(T)\geqslant k\pf}.
\end{multline*}
Write $\mathfrak N$ for a Poisson random variable with parameter $nQT$. Since $\tilde N(T)$ is bounded by $\mathfrak N$, by the Markov inequality we have that
\[
\sup_{r\in\N}\sum_{k\in\N}\sqrt{\mathbb P\po \tilde N(T)\geqslant k\pf} \leqslant \sqrt{\E\po \mathfrak N^4 \pf}\sum_{k\in\N}\frac{1}{k^2}<\infty,
\]
and hence:
\begin{multline*}
\E \po \|\delta_{Y^{r}} - \bar \eta \|_{L^1_{TV}} \pf \\ \leqslant 2T\mathbb P\po \po \mathcal A_r(T)\cap B_r(T)\pf^c \pf + 2 \sup_{k\in \N}\E\po \po\bar\tau_k - \sigma_k\pf^2 \middle|\sigma_k \leqslant T \pf^{\frac{1}{2}}\E\po\mathfrak N^2\pf^{\frac{1}{2}} \E\po\mathfrak N^4\pf^{\frac{1}{2}}\pi^2/6.
\end{multline*}
This yields that
\[
\lim_{r\rightarrow \infty} \E \po \|\delta_{Y^{r}} - \bar \eta \|_{L^1_{TV}} \pf= 0,
\]
where we used Step~4 to control the first term, which concludes the proof. 
\end{proof}

\section{Proof of Theorem~\ref{thm:soft-kill-2}}\label{sec:rate-bounded-var}

The aim of this section is to prove Theorem~\ref{thm:soft-kill-2}, and we therefore suppose that Assumption~\ref{assu:assu2} holds. In this case, we can write:
\[
\lambda_r(x) = \underline\lambda_r + m_r(x),
\]
where 
$m_r\geqslant 0$ and
\[
\|m\|_{\infty} := \sup_{r\in\N} \sup_{x\in D} m_r(x) < \infty.
\]
We start by showing that $\pi_t$ converges to a Dirac mass. Let us recall the notation~\eqref{eq:empiri-meas}: for all $t \ge0$ and $z\in D$,
\begin{equation}
 \pi_t(z) =\pi(\X_t^r)(z).
\end{equation}

\begin{lem}\label{lem:high-corr}
Under Assumption~\ref{assu:assu2}, we have for all $t>0$:
\[
\lim_{r\rightarrow \infty}\sum_{x\neq y\in D}\E(\pi_t(x)\pi_t(y)) \rightarrow 0, \qquad \lim_{r\rightarrow \infty} \E( \|\pi_t\|_{\infty}) = 1.
\]
Moreover, if $\pi(\X_0)=\delta_z$ for some $z\in D$ and all $r\in\N$, then:
\[
\sum_{x\neq y\in D}\E(\pi_t(x)\pi_t(y)) \leqslant \po Q +\frac{n_r}{2(n_r-1)}\|m\|_{\infty}\pf \frac{n_r}{\underline\lambda_r}, 
\]
and
\[
\E\po \|\pi_t\|_{\infty} \pf \geqslant 1 - \po Q + \frac{n_r}{2(n_r-1)}\|m\|_{\infty} \pf\frac{n_r}{\underline\lambda_r}.
\]
\end{lem}

\begin{proof}
Write, for a given probability measure $\xi$ on $D$,
\begin{equation}\label{eq:g2}
    g_2(\xi) = \sum_{x\in D}\xi(x)^2.
\end{equation}
We have:
\begin{align*}
Lg_2(\xi) &= \sum_{x\neq y} n_r\xi(x) \left( q(x,y) + \lambda_r(x) \frac{n_r}{n_r-1}\xi(y)\right) \\ & \qquad \qquad\left( \left( \xi(x)-\frac{1}{n_r}\right)^2 - \xi(x)^2 + \left( \xi(y)+\frac{1}{n_r}\right)^2 - \xi(y)^2 \right) \\ &= 2\sum_{x\neq y} n_r\xi(x) \left( q(x,y) + \lambda_r(x) \frac{n_r}{n_r-1}\xi(y)\right)\left( \frac{1}{n_r^2} + \frac{\xi(y)-\xi(x)}{n_r} \right) \\ &\geqslant \frac{2}{n_r}\sum_{x\neq y} \lambda_r(x) \xi(x)\xi(y) - 2\sum_{x\neq y}\xi(x)^2 q(x,y) + \frac{2n_r}{n_r-1}\sum_{x\neq y}\lambda_r(x)\xi(x)\xi(y)(\xi(y)-\xi(x)) \\ &\geqslant \frac{2\underline \lambda_r}{n_r}(1-g_2(\xi)) - 2Qg_2(\xi) - \frac{n_r}{n_r-1}\sum_{x\neq y} (\lambda_r(y)-\lambda_r(x))\xi(x)\xi(y)(\xi(y)-\xi(x)) \\ &\geqslant \frac{2\underline \lambda_r}{n_r}(1-g_2(\xi)) - \po 2Q + \frac{n_r}{n_r-1}\|m\|_{\infty}\pf.
\end{align*}
Using the Kolmogorov equation, this yields that for all $t\geqslant 0$:
\[
\partial_t \E(g_2(\pi_t)) \geqslant \frac{2\underline \lambda_r}{n_r}(1-\E(g_2(\pi_t)) ) - \po 2Q + \frac{n_r}{n_r-1}\|m\|_{\infty}\pf 
\]
and hence
\begin{equation}\label{eq:estim_g2}
\E(g_2(\pi_t)) \geqslant 1 - e^{-\frac{2\underline\lambda_r}{n_r}t} (1-\E(g_2(\pi_0))) - \po 2Q + \frac{n_r}{n_r-1}\|m\|_{\infty}\pf  \frac{n_r}{2\underline\lambda_r}.
\end{equation}
Now, since $\sum_{x\in D}\xi(x) = 1$ for all $\xi\in\mathcal M^1(D)$, we have:
\[
g_2(\pi_t) \leqslant \|\pi_t\|_{\infty},
\]
which implies the result on $\|\pi_t\|_{\infty}$.
We also have:
\[
\sum_{x\neq y\in D}\E(\pi_t(x)\pi_t(y)) = 1 - \E(g_2(\pi_t)), 
\]
which concludes the proof.
\end{proof}

The preceding Lemma showed that $\pi_t$ converges towards a Dirac measure as $r\rightarrow\infty$. The goal of the next Lemma is to determine the law of this limiting Dirac mass.

\begin{lem}\label{lem:conv-l1-expe}
 Under Assumption~\ref{assu:assu2}, we have for all $t>0$:
\[
\lim_{r\rightarrow 0} \sum_{x\in D} \left| \E(\pi_t(x)) - \mathbb P\po Y_t =x \pf \right| = 0,
\]
where $(Y_t)_{t \ge 0}$ is the Markov process with jump rates $(q(x,y))_{x,y\in D}$ and initial condition $\eta$. Moreover, if $\pi(\X^r_0)=\delta_z$ for some $z\in D$ and all $r\in\N$, then for all $T>0$, there exists $C>0$ such that for all $t \in [0,T]$:
\[
\sum_{x\in D} \left| \E(\pi_t(x)) - \mathbb P\po Y_t =x \pf \right| \leqslant C\frac{n_r}{\underline \lambda_r}.
\]
\end{lem}

\begin{proof}
The Kolmogorov equation yields that for all $x\in D$:
\begin{align*}
\partial_t\E(\pi_t(x)) &= \sum_{y\neq x} \E(\pi_t(y))q(y,x) - \E(\pi_t(x))\sum_{y\neq x}q(x,y) \\&\quad + \frac{n_r}{n_r-1}\sum_{y\neq x}\E(\pi_t(x)\pi_t(y))\left( \lambda_r(y) - \lambda_r(x) \right).
\end{align*}
Let $(\nu_t)_{t \geqslant 0}$ be the unique solution to the following system:
\[
\partial_t \nu_t(x) = \sum_{y\neq x} \nu_t(y) q(y,x) - \nu_t(x)\sum_{y\neq x}q(x,y) =L_m^*(\nu_t)(x),
\]
with initial condition $\nu_0 = \eta$. We can identify $\nu_t(x) = \mathbb P_{\eta}\po Y_t =x \pf$. Besides,
$L_m^*$ is a linear operator, and we have:
\begin{align*}
\partial_t  \left(\E(\pi_t) - \nu_t \right) &= L_m^*\left(\E(\pi_t) - \nu_t\right) + g_t,
\end{align*}
where
\[
g_t(x) = \sum_{y\neq x} \frac{n_r}{n_r-1}\E(\pi_t(x)\pi_t(y))\left( \lambda_r(y) - \lambda_r(x) \right),
\]
for all $x\in D$. We get:
\[
\E(\pi_t) - \nu_t =e^{tL_m^*}\po \pi\po \X_0^r \pf - \eta\pf +  \int_0^t e^{s L_m^*}g_{t-s} \dd s.
\]
We have that for all $t\geqslant0$:
\[
\|g_t\|_{1} \leqslant \frac{n_r}{n_r-1}\|m\|_{\infty}\sum_{x\neq y\in D} \E(\pi_t(x)\pi_t(y)) \leqslant 2\|m\|_{\infty} \po 1 - \E\po g_2(\pi_t ) \pf \pf,
\]
where $g_2$ is defined in~\eqref{eq:g2}. Since $L_m$ is a Markov generator, Lemma~\ref{lem:high-corr} (and more precisely Equation~\eqref{eq:estim_g2}) yields that
\[
\lim_{r\rightarrow \infty} \|\E(\pi_t) - \nu_t\|_{1} = 0,
\]
as well as, if $\eta = \pi\po X_0\pf$:
\[
\| \E(\pi_t) - \nu_t \|_{1} \leqslant C\frac{n_r}{\underline \lambda_r},
\]
for all $0<t<T$ and some $C>0$, which concludes the proof.
\end{proof}

We are now in position to conclude the proof of Theorem~\ref{thm:soft-kill-2}.
\begin{proof}[Proof of Theorem~\ref{thm:soft-kill-2}]
Let $F:\M^1(D)\to \R$ bounded and Lipschitz continuous with respect to $\|\cdot\|_{TV}$, and $0<\varepsilon<1/2$. Write
\[
\left| \E(F(\pi_t)) - \sum_{z\in D} F(\delta_z)\mathbb P\left( \|\pi_t-\delta_z\|_{TV} \leqslant \varepsilon \right) \right| \leqslant \sum_{z\in D} \varepsilon\|F\|_{\mathrm{Lip}}\mathbb P\left( \|\pi_t-\delta_z\|_{TV} \leqslant \varepsilon \right) +  R_{r,t},
\]
where 
\[ R_{r,t} = \|F\|_{\infty} \mathbb P\left( d( \pi_t, \Delta)>\varepsilon \right).
\]
Using the fact that $\|\pi_t-\delta_z\|_{TV} = 2\left(1-\pi_t(z)\right)$, we get
\[
R_{r,t} \leqslant \|F\|_{\infty} \mathbb P\left( \| \pi_t\|_{\infty} < 1 - \varepsilon/2 \right) \leqslant \frac{2\|F\|_{\infty}}{\varepsilon} \E(1-\| \pi_t\|_{\infty}),
\]
and finally
\begin{equation}\label{eq:conv-assu2-borne-1}
    \left| \E(F(\pi_t)) - \sum_{z\in D} F(\delta_z)\mathbb P\left( \|\pi_t-\delta_z\|_{TV} \leqslant \varepsilon \right) \right| \leqslant  \varepsilon \|F\|_{\mathrm{Lip}}+ \frac{2}{\varepsilon} \|F\|_{\infty} \E\left(1-\| \pi_t\|_{\infty}\right).
\end{equation}
Using Markov inequality, we have:
\[
\mathbb P\left(\|\pi_t-\delta_z\|_{TV} \leqslant \varepsilon \right) = \mathbb P\left( \pi_t(z) \geqslant 1-\varepsilon/2 \right) \leqslant (1-\varepsilon/2)^{-1}\E(\pi_t(z)) \leqslant (1+\varepsilon)\E(\pi_t(z)),
\]
so that
\[
    \mathbb P\left(\|\pi_t-\delta_z\|_{TV}\leqslant \varepsilon \right) - \E(\pi_t(z)) \leqslant \varepsilon \E(\pi_t(z)).
\]
Besides, we have that
\[
\E(\pi_t(z)) = \E\po \pi_t(z) \mathbbm 1_{\pi_t(z)\geqslant 1-\varepsilon/2} + \mathbbm 1_{\pi_t(z)\leqslant 1-\varepsilon/2} \pf \leqslant \mathbb P\left( \pi_t(z) \geqslant 1-\varepsilon/2 \right) + \E\po \pi_t(z)  \mathbbm 1_{\pi_t(z)\leqslant 1-\varepsilon/2} \pf,
\]
so that
\[
    \E(\pi_t(z)) - \mathbb P\left(\|\pi_t-\delta_z\|_{TV}\leqslant \varepsilon \right) \leqslant \E\po \pi_t(z)  \mathbbm 1_{\pi_t(z)\leqslant 1-\varepsilon/2} \pf.
\]
Moreover, 
\begin{align*} 
\sum_{z\in D} \E\po \pi_t(z)  \mathbbm 1_{\pi_t(z)\leqslant 1-\varepsilon/2} \pf &= \sum_{z\in D} \E\po \pi_t(z)  \mathbbm 1_{\pi_t(z)\leqslant 1-\varepsilon/2}\po \mathbbm 1_{\|\pi\|_{\infty}\leqslant 1-\varepsilon/2} + \mathbbm 1_{\|\pi\|_{\infty}> 1-\varepsilon/2}\pf\pf \\ & \leqslant \mathbb P\left( \| \pi(\X_t)\|_{\infty} < 1 - \varepsilon/2 \right) + \frac{\varepsilon}{2},
\end{align*}
and therefore
\begin{align*}
    &\left|  \sum_{z\in D} F(\delta_z)\mathbb P\left( \|\pi_t-\delta_z\|_{TV} \leqslant \varepsilon \right) - \sum_{z\in D} F(\delta_z) \E\po \pi_t(z) \pf \right| \\ & \qquad \leqslant \|F\|_{\infty} \sum_{z\in D}\left| \mathbb P\left( \|\pi_t-\delta_z\|_{TV} \leqslant \varepsilon \right) - \E\po \pi_t(z) \pf \right| 
\end{align*}
yielding that
\begin{equation}\label{eq:conv-assu2-borne-2}
    \left|  \sum_{z\in D} F(\delta_z)\mathbb P\left( \|\pi_t-\delta_z\|_{TV} \leqslant \varepsilon \right) - \sum_{z\in D} F(\delta_z) \E\po \pi_t(z) \pf \right| \leqslant \|F\|_{\infty}\po \frac{3}{2}\varepsilon + \frac{2}{\varepsilon} \E(1-\| \pi_t\|_{\infty}) \pf.
\end{equation}
Lastly, we have that:
\begin{equation}\label{eq:conv-assu2-borne-3}
\left|  \sum_{z\in D} F(\delta_z) \E\po \pi_t(z) \pf -  \sum_{z\in D} F(\delta_z) \mathbb P\po Y_t = z \pf \right| \leqslant \|F\|_{\infty}\sum_{z\in D} \left| \E(\pi_t(z)) - \mathbb P\po Y_t = z \pf \right|.  
\end{equation}
Equations~\eqref{eq:conv-assu2-borne-1},~\eqref{eq:conv-assu2-borne-2} and~\eqref{eq:conv-assu2-borne-3} together yield
\begin{multline}\label{eq:conv-assu2-borne-tot}
    \left| \E(F(\pi_t)) - \sum_{z\in D} F(\delta_z) \mathbb P\po Y_t = z \pf \right| \\ \leqslant \varepsilon \|F\|_{\mathrm{Lip}} +  \|F\|_{\infty}\po \frac{3}{2}\varepsilon + \frac{4}{\varepsilon} \E(1-\| \pi_t\|_{\infty})+\sum_{z\in D} \left| \E(\pi_t(z)) - \mathbb P\po Y_t = z \pf \right|\pf,
\end{multline}
and Lemma~\ref{lem:high-corr} and Lemma~\ref{lem:conv-l1-expe} yield, for $t>0$,
\[
\limsup_{r\rightarrow\infty} \left| \E(F(\pi_t)) - \sum_{z\in D} F(\delta_z) \mathbb P\po Y_t = z \pf \right| \leqslant \po\|F\|_{\mathrm{Lip}} + \frac{3}{2}\|F\|_{\infty}\pf\varepsilon.
\]
Letting $\varepsilon\rightarrow 0$ yields the first part of the theorem. In the case where the initial condition is a Dirac mass, plugging $\varepsilon = \sqrt{n_r/\underline \lambda_r}$ in equation~\eqref{eq:conv-assu2-borne-tot} and Lemma~\ref{lem:conv-l1-expe} yields for all $T>0$ and $0<t<T$:
\begin{align*}
&\left| \E(F(\pi_t)) - \sum_{z\in D} F(\delta_z) \mathbb P\po Y_t = z \pf \right| \\
&\quad \leqslant \|F\|_{\mathrm{Lip}} \sqrt{\frac{n_r}{\underline\lambda_r}} +  \|F\|_{\infty}\po \frac{3}{2}\sqrt{\frac{n_r}{2\underline\lambda_r}} + 4\sqrt{\frac{\underline\lambda_r}{n_r}}\po Q + \frac{n_r}{2(n_r-1)}\|m\|_{\infty}\pf \frac{n_r}{\underline\lambda_r} + C\frac{n_r}{\underline\lambda_r} \pf \\ 
&\quad \leqslant C' \po \|F\|_{\mathrm{Lip}} + \|F\|_{\infty}\pf \sqrt{\frac{n_r}{\underline\lambda_r}},
\end{align*}
for some $C'>0$, and this concludes the proof. 
\end{proof}

\section*{Acknowledgements} The work of T.L. is partially funded by the European Research Council (ERC) under the European Union’s Horizon 2020 research and innovation
programme (project EMC2, grant agreement No 810367), and also from the Agence Nationale de la Recherche through
the grant ANR-19-CE40-0010-01 (QuAMProcs). The work of J.R. is partially supported by the Agence Nationale de la Recherche through the grants  ANR-19-CE40-0010-01 (QuAMProcs) and ANR-23-CE40-0003 (Conviviality). 

\bibliography{bibliographie}
\bibliographystyle{plain}

\end{document}